\documentclass[12pt]{amsart}
\usepackage{euscript,amssymb,xr}

\let\oldmarginpar\marginpar
\renewcommand\marginpar[1]
{\oldmarginpar{\tiny\bf \begin{flushleft} #1 \end{flushleft}}}

\input xy
\xyoption{all}

\newcommand{\la}{\langle}
\newcommand{\ra}{\rangle}

\newtheorem{theorem}{\bf Theorem}[section]
\newtheorem{lemma}[theorem]{\bf Lemma}

\newtheorem{remark}[theorem]{\bf Remark}

\newtheorem{corollary}[theorem]{\bf Corollary}

\newtheorem{example}[theorem]{\bf Example}
\newtheorem{question}[theorem]{\bf Question}

\newcommand{\CC}{{\Bbb C}}
\newcommand{\RP}{{\Bbb RP}}

\newcommand{\NN}{{\Bbb N}}

\newcommand{\QQ}{{\Bbb Q}}
\newcommand{\RR}{{\Bbb R}}

\newcommand{\ZZ}{{\Bbb Z}}

\newcommand{\ggreat}{>\kern-.7ex>}
\newcommand{\ssmall}{<\kern-.7ex<}
\newcommand{\qu}{/\kern-.7ex/}
\newcommand{\exh}{\to\kern-1.8ex\to}

\newcommand{\nN}{{\EuScript{N}}}

\newcommand{\pP}{{\EuScript{P}}}
\newcommand{\sS}{{\EuScript{S}}}

\newcommand{\GL}{\operatorname{GL}}

\newcommand{\ab}{\operatorname{ab}}

\newcommand{\Aut}{\operatorname{Aut}}

\newcommand{\Diff}{\operatorname{Diff}}
\newcommand{\End}{\operatorname{End}}

\newcommand{\Ham}{\operatorname{Ham}}

\newcommand{\Homeo}{\operatorname{Homeo}}
\newcommand{\Id}{\operatorname{Id}}

\newcommand{\Isom}{\operatorname{Isom}}

\newcommand{\Ker}{\operatorname{Ker}}

\newcommand{\Mor}{\operatorname{Mor}}

\newcommand{\Perm}{\operatorname{Perm}}

\newcommand{\SO}{\operatorname{SO}}
\newcommand{\SU}{\operatorname{SU}}

\newcommand{\Stab}{\operatorname{Stab}}

\newcommand{\Symp}{\operatorname{Symp}}

\newcommand{\wt}{\widetilde}

\newcommand{\ov}{\overline}

\newcommand{\discsym}{\operatorname{disc-sym}}

\title[Actions of large finite groups on manifolds]
{Actions of large finite groups on manifolds}
\author{Ignasi Mundet i Riera}
\address{Facultat de Matem\`atiques i Inform\`atica\\
Universitat de Barcelona\\
Gran Via de les Corts Catalanes 585\\
08007 Barcelona \\
Spain}
\email{ignasi.mundet@ub.edu}

\date{November 30, 2023}


\begin{document}

\maketitle


\begin{flushright}
{\it Dedicated with gratitude to Oscar Garc\'{\i}a--Prada, \\ on the occasion of his 60th birthday}
\end{flushright}

\begin{abstract}
In this paper we survey some recent results on
actions of finite groups on topological manifolds.
Given an action of a finite group $G$ on a manifold $X$,
these results provide information
on the restriction of the action to a
subgroup of $G$ of index bounded above by a number depending only on $X$.
Some of these results refer to the
algebraic structure of the group, such as being abelian, or nilpotent, or admitting
a generating subset of controlled size; other results refer to the geometry of the action,
e.g. to the existence of fixed points, to the collection of stabilizer subgroups, or to the
action on cohomology.
\end{abstract}

\section{Introduction}

\subsection{Some questions on actions of finite groups}
Let us begin by recalling the most basic definitions of finite transformation groups.
Standard references in the field are \cite{AP,Bo,Bredon,Hsiang,tD}.

Let $X$ be a topological space and let $\Homeo(X)$ denote the group of self homeomorphisms of $X$.
A continuous action of a finite group $G$ on $X$ is a group homomorphism $\rho:G\to \Homeo(X)$.
This is usually described in terms of the map $G\times X\to X$ sending $(g,x)$ to $g\cdot x:=\rho(g)(x)$. An action $\rho:G\to\Homeo(X)$ is said to be effective if $\rho$ is injective.
Given an action of $G$ on $X$, the stabilizer of a point $x\in X$ is $G_x:=\{g\in G\mid g\cdot x=x\}$.
The action is said to be free if $G_x=\{1\}$ for every $x\in X$. A point $x\in X$ is said to be fixed
if $G_x=G$. The set of fixed points of the action of $G$ on $X$ is denoted by $X^G$.
For each $g\in G$ we denote $X^g=\{x\in X\mid g\cdot x=x\}$.

In this survey we only consider actions of finite groups on topological
manifolds. Given a topological manifold $X$,
many questions come naturally to mind regarding
the actions of finite groups on $X$.
Which finite groups act effectively on $X$? In particular, how much does the assumption
that a finite group $G$ acts effectively on $X$ prescribe the algebraic structure of $G$?
Are there natural constraints on $X$ or on the action of $G$ that force
$G$ to be abelian or nilpotent?
Is there a bound on the minimal size of a generating subset of $G$?
Does $G$ admit a free action on $X$? Does every action have
a fixed point? What are the possible collections of stabilizers of an action of $G$ on $X$?

Answering most of the previous questions in all generality, for an arbitrary $X$,
is probably out of reach with the currently available tools.
The questions become more accessible if one restricts to particular examples of manifolds.
Hence, many results in the literature aim to understand finite group actions on restricted collections
of manifolds such as spheres (see e.g. \cite{MWD,Hambleton,Sch}),
Euclidean spaces, homogeneous spaces, low dimensional manifolds
(see e.g. \cite{Edmonds1985,Edmonds2018}), or products
of them. Another possibility is to focus on very particular
examples of finite groups, such as finite cyclic groups or finite $p$-groups.
Smith theory, for example, applies to actions of finite cyclic groups of prime order on contractible manifolds or spheres. It can be extended to actions of finite $p$-groups on arbitrary manifolds using equivariant cohomology (see e.g. \cite{Bo}), but results of Jones
\cite{Jones} imply that Smith theory cannot be extended beyond $p$-groups.

Yet another strategy to make the previous questions affordable is to consider
actions of a group $G$ on a manifold $X$
and to prove properties, not on the action of $G$ on $X$,
but on the restriction to some subgroup of $G$ of index
bounded above by a constant depending only on $X$.
All the results  in this survey follow this strategy. As a consequence, they don't say anything
interesting on actions of small finite groups, but they become meaningful once one
considers actions of large finite groups (where the meaning of {\it large} depends
on the manifold supporting the action). The benefit of allowing to pass to a subgroup
is, as we will see, that the results are valid for large collections of manifolds, and in some cases for
manifolds satisfying only a finiteness condition such as being
closed, or compact, or having finitely generated integral homology.

If a topological manifold $X$ is endowed with some geometric structure then one may consider
actions of finite groups $G\to\Homeo(X)$ whose image is contained in the group $\Aut(X)\leq\Homeo(X)$ of homeomorphisms
preserving the given structure. For example, we may consider differentiable, complex
or symplectic structures.
Although our main focus is on continuous actions,
we will say a few words on actions preserving geometric structures. The automorphism group of
a geometric structure on $X$ is usually
much smaller than $\Homeo(X)$, so questions on finite transformation groups tend to become
simpler in the presence of invariant geometric structures, as we will see in a few examples.

\subsection{Conventions, notations and contents.}
In this paper {\it manifold} means topological manifold, possibly with boundary.
A closed manifold is a compact manifold with empty boundary.
By convention, all group actions on manifolds will be implicitly assumed to be
continuous. If $X$ and $Y$ are topological spaces, $X\cong Y$ will mean that
$X$ and $Y$ are homeomorphic. When we refer to a $p$-group without specifying the prime
$p$ we mean a $p$-group for an arbitrary prime $p$.
If $G$ is a group, by $H\leq G$ we mean that $H$ is a subgroup of $G$.

If $P(n)$ denotes some statement depending on a natural number $n$, we say that
{\it $P(n)$ is true for arbitrarily large values of $n$}
if there exists a sequence of natural numbers,
$n_i\to\infty$, such that $P(n_i)$ is true for every $i$.
Note that it may certainly happen that
both $P(n)$ and its negation are true for arbitrary large values of $n$.

Most of the results stated in this survey have appeared with proof elsewhere, so we
won't prove them here. The main exceptions are Theorems
\ref{thm:disc-sym-low-dimension},
\ref{thm:smooth-disc-sym-low-dimension} and
\ref{thm:nilpotent-generators}, which
are proved in Section \ref{s:proofs},
the part of Theorem \ref{thm:disc-sym-rationally-hypertoral-WLS}
referring to weak Lefschetz symplectic manifolds, which will be proved in
\cite{Mundet2023-0}, and Theorem \ref{thm:abelian-free-actions-tori}, which will be proved in \cite{Mundet2023-1}.

Each section except for the last one is concerned with a particular aspect of finite group
actions. Sections from \ref{s:Jordan} to \ref{s:almost-asymmetric}
refer to the group itself (how far it may be from being abelian or nilpotent, how many elements
you need to generate it, or how big it can be), while Sections from \ref{s:trivial-homology} to \ref{s:stabilizers} refer to the geometry of the action (induced action on homology, the rotation morphism, existence of fixed points, and number of
different stabilizer subgroups).

\section{Ghys's question and Jordan property}
\label{s:Jordan}

Around twenty years ago  \'Etienne Ghys asked in a series of talks \cite{Ghys} the following question:
given a closed manifold $X$, does there exists a constant $C$ such that any finite group $G$
acting effectively on $X$ has an abelian subgroup $A\leq G$ satisfying $[G:A]\leq C$?
Ghys was most probably thinking about a {\it differentiable} manifold and about {\it smooth} actions
on it (see \cite[Question 13.1]{Fisher}), as one of his motivations was \cite{Ghys1993}, but
the question also makes sense for continuous actions on (topological) manifolds.

Motivated by a similar conjecture by Jean-Pierre Serre \cite{Serre} on the Cremona group and its natural extension
to arbitrary birational transformations groups, Vladimir L. Popov \cite{Po0} defined a group $\Gamma$ to be
{\it Jordan} if there is a number $C$ such that any finite subgroup $G\leq \Gamma$
has an abelian subgroup $A\leq G$ satisfying $[G:A]\leq C$.
The name is inspired by Jordan's theorem \cite[\S 36]{CR},
according to which $\GL(n,\CC)$ is Jordan for every $n$
(the result originally appeared in \cite{Jordan}; see \cite{Breu}
for a modern exposition of Jordan's original argument and for
references to alternative proofs). So Ghys's question
asks whether $\Diff(X)$ is Jordan for every closed manifold $X$.

Partial positive answers to Ghys's question appeared in \cite{Mundet2010,Mundet2016,Mundet2019,Zimmermann2014}.
The results in \cite{Mundet2019} are based on the main result in \cite{MundetTurull},
which characterizes Jordan groups in terms of finite
subgroups whose cardinal is of the form $p^aq^b$ for primes $p,q$ and integers $a,b$. The paper \cite{MundetTurull} uses fundamentally the classification of finite simple groups (CFSG).

In 2014 Bal\'azs Csik\'os, L\'aszl\'o Pyber and Endre Szab\'o \cite{CPS} proved
that $\Diff(T^2\times S^2)$ is not Jordan, thus giving
the first example of a smooth manifold whose diffeomorphism group is not Jordan.
This was followed by more examples in \cite{Mundet2017,Szabo2019}. Consequently,
Ghys modified his conjecture replacing {\it abelian} by {\it nilpotent} \cite{Gh2}. This modified
conjecture was proved
in dimension four by the author and Carles S\'aez-Calvo \cite{MundetSaez}, and it has been recently
proved in arbitrary dimensions for continuous actions on (topological) manifolds by
Csik\'os, Pyber and Szab\'o \cite{CPS2}. Actually, the main result in \cite{CPS2}
does not require the manifold to be compact:

\begin{theorem}[Theorem 1.4 in \cite{CPS2}]
\label{thm:CPS2}
Let $X$ be a manifold such that $H_*(X;\ZZ)$ is finitely generated.
There exists a constant $C$ such that any finite group $G$ acting effectively on $X$
has a nilpotent subgroup $N\leq G$ satisfying $[G:N]\leq C$.
\end{theorem}

Two of the main ingredients in the proof of \cite[Theorem 1.4]{CPS2}, the first
a result on finite groups proved using the CFSG \cite[Corollary 3.18]{CPS2}, and the second
a result on finite transformation groups \cite[Lemma 6.1]{CPS2}, can be combined to obtain
a criterion for Jordan property {\it of homeomorphism groups} which only requires
to consider finite $p$-groups. This has allowed to extend the results
on Jordan property for diffeomorphism groups proved in \cite{Mundet2019} to homeomorphism groups.
The current knowledge on the question is summarized in the following theorem.

\begin{theorem}
\label{thm:Jordan-homeo}
Let $X$ be a manifold. If any of the following conditions is true,
then $\Homeo(X)$ is Jordan:
\begin{enumerate}
\item $X$ is compact and $\dim X\leq 3$, see \cite{Mundet2022} (the case $\dim X=3$ follows from combining
\cite{Pardon2019,Zimmermann2014}),
\item $X$ is $n$-dimensional and $H_*(X;\ZZ)\simeq H_*(S^n;\ZZ)$, see \cite{Mundet2022},
\item $X$ is connected, $H_*(X;\ZZ)$ is finitely generated, and the Euler characteristic
$\chi(X)$ of $X$ is nonzero, see \cite{Mundet2022},
\item $X$ is rationally hypertoral (we define this below), see \cite{Mundet2021},
\item $X$ is closed and it supports a flat metric, see \cite[Corollary 1.7]{Ye}.
\end{enumerate}
If $X$ supports an effective action of $\SU(2)$ or $\SO(3,\RR)$ then $\Homeo(T^2\times X)$
is not Jordan, see \cite{Mundet2017}.
\end{theorem}

We say that an $n$-manifold $X$ is {\it rationally hypertoral} if
$X$ is closed, connected and orientable, and it admits a map of nonzero
degree to $T^n=(S^1)^n$. If $X$ is connected,
this is equivalent to the property that $H^n(X;\RR)\simeq\RR$ and the cup product map
$\Lambda^nH^1(X;\RR)\to H^n(X;\RR)$ is surjective.

No characterisation seems to be known at present of which manifolds have Jordan homeomorphism group.

The CFSG is used both in the proof of Theorem \ref{thm:CPS2} and in the proof of
cases (2) and (3) of Theorem \ref{thm:Jordan-homeo}. Since the proof of the CFSG
is extremely long and complicated, it is natural to ask whether it can be avoided
in those proofs, perhaps replacing it by geometric arguments.

\begin{question}
Can one prove Theorem \ref{thm:CPS2} and
cases (2) and (3) of Theorem \ref{thm:Jordan-homeo} without using
the CFSG?
\end{question}

The Jordan property has  recently been studied
for automorphism groups of some geometric structures on
manifolds. In \cite{Mundet2017-0} it was proved
that if $X$ is $T^2\times S^2$ endowed with any symplectic structure then $\Symp(X)$ is
Jordan. This was extended in \cite{MundetSaez} to all closed symplectic
$4$-manifolds. The automorphism group of any closed almost complex $4$-manifold has
been proved to be Jordan in \cite{MundetSaez}. The particular case of compact complex surfaces
had been earlier proved by Yuri Prokhorov and Constantin Shramov in \cite{PS21}.
Finally, the isometry group of any closed Lorentz $4$-manifold has been proved to be Jordan in \cite{Mundet2020}.

In higher dimensions the following questions seem to be open.

\begin{question}
Let $X$ be a compact symplectic (resp. Lorentz, complex, or almost complex) manifold.
Is the automorphism group of $X$ necessarily Jordan?
\end{question}

There are however a few results on Jordan property of automorphism groups
of geometric structures in higher dimensions. In \cite{Mundet2018} it is proved that
the symplectomorphism group of any compact symplectic manifold with vanishing first
Betti number is Jordan. It is also proved in \cite{Mundet2018} that Hamiltonian
diffeomorphism groups of arbitrary compact symplectic manifolds are Jordan.
Sheng Meng and De-Qi Zhang proved in \cite{MZ} that the automorphism group of any
projective variety defined in characteristic zero is Jordan.
Automorphism groups of compact Kaehler manifolds have been
proved to be Jordan by Jin Hong Kim in \cite{Kim}. This has been extended to
compact varieties in Fujiki's class (that is, compact reduced complex spaces which
are bimeromorphic to a compact K\"ahler manifold) by Sheng Meng, Fabio Perroni and De-Qi Zhang
in \cite{MPZ}. Recently, Aleksei Golota has proved that parallelizable
compact complex manifolds have Jordan automorphism group in \cite{Golota2}.
Other interesting results on Jordan property for automorphism groups of certain complex
manifolds have been obtained by F. Bogomolov, N. Kurnosov, A. Kuznetsova and E. Yasinsky
for Bogomolov--Guan manifolds \cite{BKKY22},
by K. Loginov for manifolds of small Kodaira dimension \cite{Log22},
and by A. Savelyeva for Hopf manifolds \cite{Sav24}.

We mentioned at the beginning of this section
that the Jordan property has also been studied in algebraic geometry for birational transformation
groups. The situation there is remarkably parallel to that for homeomorphism
groups, with some varieties having Jordan birational transformation group and others not (in fact,
the first counterexample to be found is the product of an elliptic curve by the projective line,
see \cite{Zarhin}).
Some central contributions to these question are, among others, work of Jean-Pierre Serre \cite{Serre},
Vladimir Popov \cite{Po0}, Yuri G. Zarhin \cite{Zarhin}, and Yuri Prokhorov and Constantin Shramov
\cite{PS14,PS16}.
The interested reader can consult the paper \cite{Guld} by Attila Guld for an analogue
in that context of Theorem \ref{thm:CPS2}, a list of references, and the history of the
problem.

\section{Actions of finite abelian groups}
\label{s:abelian}

For any finite group $G$ we denote by $d(G)$ the minimal size of a generating set of $G$.
By convention, the trivial group has $d=0$.

The next result follows from a theorem of L.N. Mann and J.C. Su
\cite[Theorem 2.5]{MannSu}.

\begin{theorem}
\label{thm:MS}
Let $X$ be a compact manifold. There exists a constant $C$, depending only on $X$,
such that for any finite abelian group $A$ we have $d(A)\leq C$.
\end{theorem}

The original theorem of Mann and Su refers to actions of groups of the form $(\ZZ/p)^k$.
The bound in Theorem \ref{thm:MS} is the same one as that in
\cite[Theorem 2.5]{MannSu} and is explicit: it can be chosen to depend on the sum of the Betti
numbers of $X$ maximized over all possible fields. The theorem of Mann and Su,
and hence Theorem \ref{thm:MS},
is also valid for non compact manifolds $X$ with finitely generated $H_*(X;\ZZ)$ (see \cite[Theorem 1.8]{CMPS}).

Finding, for an arbitrary manifold $X$, the optimal value of the constant $C(X)$
in Theorem \ref{thm:MS} is probably a very difficult question. Following the strategy
described in the Introduction we are led to the following questions, which we will motivate below.

\begin{question}
\label{quest:abelian-1}
Does there exist, for every compact connected $n$-dimensional manifold $X$, a constant $C$ such that any
finite abelian group $A$ acting effectively on $X$ has a subgroup
$B$ satisfying $[A:B]\leq C$ and $d(B)\leq n$?
\end{question}

\begin{question}
\label{quest:abelian-1-2}
Is $X=T^n$ the only $n$-dimensional compact connected manifold for which the bound $n$ in the previous question
cannot be replaced by $n-1$?
\end{question}

Answering affirmatively Question \ref{quest:abelian-1-2} would materialize in the context of finite group actions on manifolds a beautiful poem of Tom\`as Garc\'es\footnote{"Si veiessis el blau fumerol / adormit a la vella teulada, / ?`em diries quants rostres hi ha / en el clar borrissol de la flama? / Si un vaixell, en el tr\`emul mat\'{\i} / de les lloses del moll se separa, / ?`no sabries, per l'ombra que es mou, / els ad\'eus que bateguen en l'aire? / I si un trot fugitiu deixond\'{\i} / els camins esva\"{\i}ts de la tarda, / ?`per l'espurna que fan els cavalls, / endevines els ulls de la dama?". This is a rough translation, lacking unfortunately the musicality of the original Catalan version: "If you saw the blue sleeping smoke / lying down on the ancient roof / would you tell me the number of faces / on the clear fluff of the flame? / If a boat, in the trembling morning / goes away from the quay / would you know, from the moving shadow / the beating farewells in the air? / And if a fleeing gallop awoke / the fading paths of the afternoon / would the sparks made by
the horse / tell you the eyes of the lady?".}: it would have the remarkable consequence that, in the particular case of the torus, you can recover a manifold from information on the collection of finite groups that act on it. This is usually impossible, as there are plenty of closed asymmetric manifolds (see Section \ref{s:almost-asymmetric} below).

The optimal value, for a fixed manifold $X$, of the constant $C$ in Question \ref{quest:abelian-1}
can be arbitrarily large as soon as $n\geq 2$. For example, any finite group acts freely (hence effectively) on some closed connected and orientable surface; applying this fact to
$(\ZZ/p)^r$ for any prime $p$ and any $r>2$ we get closed connected orientable surfaces for which the optimal constant $C$ in Question \ref{quest:abelian-1} is as large as we wish.

We next explain the motivation behind the previous questions.
Let $X$ be a compact $n$-manifold, and consider the set
$$\mu(X)=\{m\in\NN\mid \text{$X$ supports effective actions of $(\ZZ/r)^m$ for arbitrarily large $r$}\}.$$
Theorem \ref{thm:MS} implies that $\mu(X)$ is finite. Following \cite{Mundet2021} we define
the {\it discrete degree of symmetry of $X$} to be
$$\discsym(X)=\max (\{0\}\cup\mu(X)).$$
By \cite[Lemma 2.7]{Mundet2021}, for any nonnegative integer $k$ the inequality $\discsym(X)\leq k$
is equivalent to the existence of a contant $C$ such that any finite abelian group $A$
acting effectively on $X$ has a subgroup $B$ satisfying $[A:B]\leq C$
and $d(B)\leq k$. Hence, the Question \ref{quest:abelian-1} (resp. Question \ref{quest:abelian-1-2})
is equivalent to the following Question \ref{quest:abelian-2} (resp. Question \ref{quest:abelian-2-2}).

\begin{question}
\label{quest:abelian-2}
Is $\discsym(X)\leq n$ for every compact connected $n$-manifold $X$?
\end{question}

\begin{question}
\label{quest:abelian-2-2}
If a compact connected $n$-manifold $X$ satisfies $\discsym(X)=n$,
is $X$ necessarily homeomorphic to $T^n$?
\end{question}

An affirmative answer to Questions
\ref{quest:abelian-2} and  \ref{quest:abelian-2-2} would be an example of a rigidity
result\footnote{By this we mean a result that fits the
following vague pattern. Let $\sS$ be a set of geometric objects of some type and let
$\phi:\sS\longrightarrow \RR$ be a map. A rigidity result for $(\sS,\phi)$ is the statement that
(1) there is an upper bound $M:=\max\phi(\sS)<\infty$, and that (2) $\phi^{-1}(M)\neq\emptyset$
and the objects in $\phi^{-1}(M)\subset \sS$ are more rigid (or less abundant) than those in other fibers of $\phi$.
Some examples: (i) $\sS=\{\text{simple smooth curves $\gamma\subset\RR^2$ of total length $2\pi$}\}$,
$\phi(\gamma)=\text{area enclosed by $\gamma$}$; here, the isoperimetric inequality
is a rigidity result.
(ii) Mostow rigidity after Besson-Courtois-Gallot, see \cite{BCG}.
(iii) $\sS=\{\text{$U(p,q)$-local systems on a closed connected surface}\}$,
$\phi=\text{Toledo invariant}$; here the rigidity result is due to Toledo, Hern\'andez, and Bradlow-Garc\'{\i}a--Prada-Gothen,
see \cite{BGPG,Hernandez,Toledo}. Answering affirmatively Questions
\ref{quest:abelian-2} and  \ref{quest:abelian-2-2} would imply a rigidity result for $\sS=\{\text{closed connected
$n$-dimensional manifolds}\}$ and $\phi=\discsym$.}.

The requirement in the previous questions that the manifold $X$ is connected is crucial,
for otherwise there would be no hope of bounding $\discsym(X)$ by a constant depending
only on $\dim X$. Indeed, if for example $X=X_1\sqcup\dots\sqcup X_k$ and each $X_i$ is equal
to $S^1$, then $X$ is one dimensional but it supports an effective action of $T^k$, where
the action of $(\theta_1,\dots,\theta_k)\in T^k$ on the component $X_j\subset X$ is given
by multiplication by $\theta_j$. In fact, we have in this case $\discsym(X)=k$, as the reader
can easily check.

Why are the previous questions reasonable? For any sequence of integers $r_i\to\infty$
it seems to be a natural intuition that the sequence of groups $(\ZZ/r_i)^m$ "converges"
to the torus $T^m$. Thus one may heuristically expect that having
effective actions of each of the groups $(\ZZ/r_i)^m$ on
a given manifold should have similar implications as
having an action of $T^m$ on that manifold.
Since no connected $n$-manifold supports a continuous action of
a torus of dimension bigger than $n$, and since the only $n$-manifold supporting
an effective action of $T^n$ is $T^n$ itself (see e.g. the proof of Theorem 1.9 in
\cite{Mundet2021} for a proof), the previous heuristic naturally leads to the question above.

One may transform the previous heuristic into an actual theorem in different ways.
Perhaps the most natural is the following one, which is a little
exercise in Lie group theory (see \cite[Theorem 1.10]{Mundet2021}).

\begin{theorem}
\label{thm:cpct-Lie-group}
Let $G$ be a compact Lie group. There exists a constant $C$ such that for every
finite abelian subgroup $A\leq G$ there is a torus $T\leq G$
satisfying $[A:A\cap T]\leq C$. Hence, the following are equivalent for each $m\in\NN$:
\begin{enumerate}
\item There exists subgroups of $G$ isomorphic to $(\ZZ/r)^m$ for arbitrarily big integers $r$.
\item There exists an $m$-dimensional torus in $G$.
\end{enumerate}
\end{theorem}

Note that even if $G$ is a connected compact Lie group not every finite abelian
subgroup of $G$ is contained in a torus. For example, if $G=\SO(3,\RR)$ and
$A\leq G$ is the subgroup of diagonal matrices with entries $\pm 1$ then $A\simeq(\ZZ/2)^2$. However every nontrivial torus $T\leq \SO(3,\RR)$ is isomorphic to
$S^1$, and hence $A$ can't possibly be contained in a torus in $\SO(3,\RR)$,
for otherwise it would be cyclic. This shows
that the constant $C$ in the previous theorem cannot be chosen to be $1$ in some
cases where $G$ is connected.

If $(X,g)$ is a closed Riemannian manifold, then by Myers--Steenrod's theorem the isometry group
$\Isom(X,g)$ is a compact Lie group. Applying Theorem \ref{thm:cpct-Lie-group} to it we conclude
that the following are equivalent:
\begin{enumerate}
\item There exist effective actions by isometries of $(\ZZ/r)^m$ on $(X,g)$ for arbitrarily
big integers $r$.
\item There exists an effective action by isometries of $T^m$ on $(X,g)$.
\end{enumerate}

If one replaces the compact Lie group $G$ by the homeomorphism group of a manifold then things
become much more complicated, and in fact the analogue of Theorem \ref{thm:cpct-Lie-group} in that
case fails to be true in general, as shown by the following theorem.

\begin{theorem}
There exist closed connected manifolds $X$
satisfying $\discsym(X)\geq 1$ but supporting no effective action of the circle.
\end{theorem}

The construction of the manifolds $X$ in the previous theorem and the proof that
they support no effective action of the circle is due to Cappell, Weinberger and Yan
\cite{CWY}. The fact that $\discsym(X)\geq 1$ was known to the authors of
\cite{CWY} (see \cite[Remark 1.3]{vL}), and a detailed proof appears in \cite[Theorem 1.11]{Mundet2021}.

\section{Bounds on the discrete degree of symmetry}
\label{s:bounds-on-discsym}

We don't know the answer to Question \ref{quest:abelian-2} in general,
but some partial positive results are available, and no counterexample has been found so far.
Before stating the positive results we introduce a new definition. An $n$-manifold $X$ is said to be {\it weak Lefschetz symplectic}
(WLS for short) if $X$ is connected, closed and orientable and in addition there exists some class $\Omega\in H^2(X;\RR)$
such that the image of the cup product map $\Lambda^*H^1(X;\RR)\otimes\RR[\Omega]\to H^*(X;\RR)$
contains $H^{n-1}(X;\RR)\oplus H^*(X;\RR)$. For example, any compact connected Kaehler manifold
is weak Lefschetz symplectic, by Lefschetz's decomposition theorem.

The following theorem combines \cite[Theorem 1.3]{Mundet2021} (for
the rationally hypertoral case) and the main result in \cite{Mundet2023-0} (for the WLS case).

\begin{theorem}
\label{thm:disc-sym-rationally-hypertoral-WLS}
Let $X$ be a closed connected $n$-dimensional manifold.
Suppose that $X$ is rationally hypertoral or WLS.
Then $\discsym(X)\leq n$. If $\discsym X=n$, then
$H^*(X;\ZZ)\simeq H^*(T^n;\ZZ)$ as rings, and the universal abelian cover of $X$
is acyclic. If
$\discsym X=n$ and the fundamental group $\pi_1(X)$ is virtually solvable, then $X\cong T^n$.
\end{theorem}

The universal abelian cover
of a path-connected space $X$ with fundamental group
$\pi$ can be identified with $X^{\ab}:=\wt{X}/[\pi,\pi]$,
where $\wt{X}$ is the universal cover of $X$. The cover $X^{\ab}$ has a residual
action of $H_1(X;\ZZ)\simeq \pi/[\pi,\pi]$, and the orbit map of this
action is a principal $H_1(X;\ZZ)$-bundle $\pi:X^{\ab}\to X$.
If $H_1(X;\ZZ)$ is torsion free and
$X$ has the homotopy type of a CW complex, one may
describe $\pi:X^{\ab}\to X$ as follows.
Take a map $\phi:X\to T^r$
such that $\phi^*:H^1(T^r;\ZZ)\to H^1(X;\ZZ)$ is an isomorphism.
($\phi$ exists by the assumption on the homotopy type of $X$.) Then
$X^{\ab}\to X$ is homeomorphic to the pullback via $\phi$ of the $\ZZ^r$-bundle $\RR^r\to\RR^r/\ZZ^r=T^r$.

\begin{remark}
The statement of the previous theorem is slightly redundant. Indeed, if the universal
abelian cover $X^{\ab}$ of an path-connected space $X$ is acyclic then $H_1(X;\ZZ)$, which
acts freely on $X^{\ab}$, cannot have nontrivial torsion, by Smith's fixed point theorem for
actions of $\ZZ/p$ ($p$ prime) on acyclic manifolds (see e.g. \cite[Chap II, Corollary 4.6]{Bo}).
Hence $X$ can be identified with
with a quotient $X^{\ab}/\ZZ^r$, where $\ZZ^r$ acts freely and proper discontinuously on
$X^{\ab}$. The projection $X^{\ab}\times_{\ZZ^r}\RR^r\to X^{\ab}/\ZZ^r$ is a homotopy
equivalence (it is a locally trivial fibration with fibers homeomorphic to $\RR^r$),
while applying Serre's spectral sequence to the projection $X^{\ab}\times_{\ZZ^r}\RR^r\to\RR^r/\ZZ^r$
we conclude that $H^*(X;\ZZ)\simeq H^*(T^r;\ZZ)$. If $X$ is a closed manifold, this can only
happen if $X$ is connected and $r=\dim X$.
\end{remark}

The following theorem combines Theorems 1.3 and 1.4 in \cite{Mundet2022}.

\begin{theorem}
\label{thm:discsym-chi-nonzero-or-sphere}
Let $X$ be a closed connected $n$-dimensional manifold.
If $\chi(X)\neq 0$ then $\discsym(X)\leq [n/2]$. If
$H^*(X;\ZZ)\simeq H^*(S^n;\ZZ)$ then $\discsym(X)\leq [(n+1)/2]$.
\end{theorem}

The following result, which is \cite[Theorem 1.2]{Mundet2021},
provides further evidence suggesting a positive answer to Question \ref{quest:abelian-2}.

\begin{theorem}
\label{thm:weak-bound-disc-sym}
For any closed connected $n$-manifold $X$ we have $\discsym(X)\leq [3n/2]$.
\end{theorem}

In low dimensions we have the following result, which is proved in Subsection
\ref{ss:thm:disc-sym-low-dimension}.

\begin{theorem}
\label{thm:disc-sym-low-dimension}
Let $X$ be a closed connected manifold of dimension $\leq 3$.
We have $\discsym(X)\leq\dim X$, with equality if and only if $X$ is homeomorphic
to a torus.
\end{theorem}

It is an interesting problem to compute or estimate the discrete degree of symmetry.
In dimensions $\leq 2$ this is easy using standard tools, but in higher dimensions this
is substantially more challenging. Here are some examples.

\begin{example}
\label{ex:computation-discsym}
\begin{enumerate}
\item $\discsym(S^1)=1$ (this is elementary).
\item $\discsym(T^2)=2$,
$\discsym(S^2)=\discsym(\RP^2)=\discsym(\RP^2\sharp\RP^2)=1$,
$\discsym((T^2)^{\sharp g})=0$ if $g\geq 2$,
and $\discsym((\RP^2)^{\sharp g})=0$ if $g \geq 3$. This follows from the fact that
any continuous finite group action on a closed surface is conjugate to a conformal
transformation for a conveniently chosen conformal structure
(K\'er\'ekjart\'o's theorems, see
e.g. \cite{ConstantinKolev} and especially the Remark at the end of
\cite{ConstantinKolev}), together with the arguments in the proof
of \cite[Theorem 1.3]{Mundet2010}.
\item Fix natural numbers $k,n$ satisfying $1\leq k\leq n-1$.
Let $\sigma:T^n\to T^n$ be the free involution defined by
$\sigma(x_1,\dots,x_n)=(x_1+1/2,\dots,x_{k}+1/2,-x_{k+1},\dots,-x_n)$.
Then $\discsym(T^n/\sigma)=k$. See \cite[Theorem 1.13]{Mundet2021}. (For example,
setting $n=2$ and $k=1$ we get the Klein bottle $\RP^2\sharp\RP^2$.)
\item Let $Z$ be a closed and connected $m$-manifold such that
$H_*(Z;\ZZ)\not\simeq H_*(S^m;\ZZ)$. For every $n\geq 0$ we have
$\discsym(T^n\times (T^m\sharp Z))=n$. See \cite[Theorem 1.5]{Mundet2021}.
\item $\discsym(S^n)=[(n+1)/2]$. This follows from Theorem \ref{thm:discsym-chi-nonzero-or-sphere}
and the fact that the torus $T^{[(n+1)/2]}$ acts effectively on $S^n$.
\end{enumerate}
\end{example}

\newcommand{\free}{\operatorname{free}}
\newcommand{\primenumber}{\operatorname{prime}}
\newcommand{\stablerank}{\operatorname{stable-rank}}
\newcommand{\stablefreerank}{\stablerank^{\free}}

A notion related to the discrete degree of symmetry is the stable rank.
Let $X$ be a  manifold. Let
$\mu_{\primenumber}(X)$ be the set of all natural numbers $m$ such that
$X$ supports effective actions of $(\ZZ/p)^m$ for arbitrarily large primes $p$.
Define the {\it stable rank of $X$} to be $\stablerank(X):=\max(\{0\}\cup\mu_{\primenumber}(X))$.
We obviously have $\stablerank(X)\leq\discsym(X)$, so a weaker version of Question \ref{quest:abelian-2}
would ask whether, for a closed and connected manifold $X$, we necessarily
have $\stablerank(X)\leq\dim X$, and a stronger version of Question \ref{quest:abelian-2-2}
would be whether the case of equality only occurs for $X=T^n$.

The discrete degree of symmetry is an analogue for actions of finite groups
of the degree of symmetry of a manifold $X$. The latter is defined
to be the maximum of the dimensions of the compact Lie groups acting effectively on $X$,
and it has been extensively studied in the literature. See for example \cite[Chap. VII, \S 2]{Hsiang}.

\section{Abelian group actions preserving geometric structures}
\label{s:abelian-smooth}

\newcommand{\smooth}{\operatorname{smooth}}

We may consider analogues of the discrete degree of symmetry
defined in terms of actions of $(\ZZ/r)^m$
preserving some geometric structure.
This leads for example to the smooth, symplectic or holomorphic
discrete degree of symmetry.

Denoting by $\discsym_{\smooth}(X)$ the smooth discrete degree of symmetry
of a smooth manifold $X$, we have the following result, which is proved in
Section \ref{ss:thm:smooth-disc-sym-low-dimension}.

\begin{theorem}
\label{thm:smooth-disc-sym-low-dimension}
Let $X$ be a closed connected smooth $n$-manifold, where $n\leq 4$. We have
$\discsym_{\smooth}(X)\leq n$. If $\discsym_{\smooth}(X)=n$ then $H^*(X;\ZZ)\simeq H^*(T^n;\ZZ)$
and the universal abelian cover of $X$ is acyclic.
If $\discsym_{\smooth}(X)=n$ and $\pi_1(X)$ is virtually solvable then $X\cong T^n$.
\end{theorem}

An interesting question is to find manifolds for which one can define
any of these notions and obtain a smaller value than the (continuous) discrete degree of symmetry.
(The question is probably substantially more difficult for smooth or symplectic structures than
for holomorphic ones.)

\begin{question}
For a manifold $X$, let $\sigma$ denote a smooth, symplectic or holomorphic structure
on $X$, and let $\discsym_{\sigma}(X)$ denote the discrete degree of symmetry defined
considering only actions of groups that preserve $\sigma$. Do there exist examples
of $X$ and $\sigma$ for which $\discsym_{\sigma}(X)<\discsym(X)$?
\end{question}

This does not seem to have been addressed so far in the literature. Some discrepancy
between continuous and smooth finite group actions has been pointed out in \cite[Theorem 1.5]{Mundet2021},
related to the existence of exotic smooth structures on tori, but this discrepancy does not give
any example of a manifold for which the smooth and the continuous discrete degrees of symmetries
differ. In contrast, there exist results on the analogue for actions of compact connected Lie groups: Amir Assadi and Dan Burghelea proved in \cite{AssadiBurghelea} that if
$\Sigma$ is an $n$-dimensional exotic sphere then $T^n\sharp\Sigma$ does not support any smooth
action of the circle, whereas it does support an effective continuous action of $T^n$, as it is homeomorphic to $T^n$ itself  (see also the references in \cite{AssadiBurghelea} for earlier
examples).

\newcommand{\Hamdiscsym}{\operatorname{Ham-disc-sym}}

If $X$ is a symplectic manifold, one can define
$\mu_{\Ham}(X)$ to be the set of all natural numbers $m$ such that the group
$\Ham(X)$ of Hamiltonian diffeomorphisms of $X$ contains subgroups
isomorphic to $(\ZZ/r)^m$ for arbitrarily large integers $r$,
and then define accordingly the {\it Hamiltonian discrete degree of symmetry of $X$} to be
$\discsym_{\Ham}(X):=\max(\{0\}\cup\mu_{\Ham}(X))$.
It follows from \cite[Corollary 1.7]{Mundet2022} and
Lemma \ref{lemma:linearising-smooth} below that if $X$ is compact then
$\discsym_{\Ham}(X)\leq \dim X/2$. The same considerations that lead to Question
\ref{quest:abelian-2} suggest the following.

\begin{question}
Let $X$ be a compact symplectic manifold. If $\discsym_{\Ham}(X)=\dim X/2$,
does $X$ necessarily support a structure of toric manifold?
\end{question}

Compact toric manifolds admit cell decompositions all of whose cells are even dimensional,
and consequently their integral cohomology is free as a $\ZZ$-module and concentrated in
even degrees. With this in mind, the following result provides evidence for the previous
question.

\begin{theorem}
\label{thm:toric-finite-subgroup}
If a compact symplectic manifold $X$ satisfies $\discsym_{\Ham}(X)=\dim X/2$, then
$H^*(X;\ZZ)$ is free as a $\ZZ$-module and $H^k(X;\ZZ)=0$ for odd $k$.
\end{theorem}
\begin{proof}
Let $X$ be a $2n$-dimensional compact symplectic manifold.
We are going to use the solution to the integral Arnold conjecture
(see \cite[Corollary 1.2]{AbouBlum}, \cite[Theorem A]{BaiXu}, or \cite[Theorem 1]{Rezchikov}),
which implies that for every prime $p$ the number of fixed points of a nondegenerate
Hamiltonian diffeomorphism is bounded below by $\dim_{\ZZ/p}H^*(X;\ZZ/p)$. This implies
that every Hamiltonian diffeomorphism has some fixed point (note that
the rational Arnold conjecture is enough for this implication).

Suppose that $\discsym_{\Ham}(X)=n$.
Then there is a sequence of integers $r_i\to\infty$ and, for each $i$, a subgroup
of $\Ham(X)$ isomorphic to $(\ZZ/r_i)^n$.
Let $\pP=\{p\text{ prime}\mid p\text{ divides $r_i$ for some $i$}\}$.
We distinguish two possibilities. If $\pP$ is infinite, then we can take a sequence
of primes $p_j$ belonging to $\pP$ and satisfying $p_j\to\infty$. Each
$p_j$ divides $r_{i_j}$ for some $i_j$, so $(\ZZ/p_j)^m$ is isomorphic to
a subgroup of $(\ZZ/r_{i_j})^m$ and hence it is also isomorphic to a subgroup of $\Ham(X)$.
The second possibility is that $\pP$ is bounded. In that case, there exists some
$p\in\pP$ and a sequence of natural numbers $e_j\to\infty$ such that $p^{e_j}$
divides $r_{i_j}$ for some $i_j$. Arguing as before, this gives a subgroup of $\Ham(X)$
isomorphic to $(\ZZ/p^{e_j})^m$ for each $j$. In conclusion, we may assume that there
is a sequence of integers $r_i\to\infty$, such that either each $r_i$ is a prime or
each $r_i$ is of the form $p^{e_i}$ for some fixed prime $p$, and
for each $i$ there is a subgroup of $\Ham(X)$ isomorphic to $(\ZZ/r_i)^n$.

By \cite[Theorem 2.3]{Mundet2022} and \cite[Lemma 2.1]{Mundet2021} we may assume
the existence of a prime power $r\geq 4$, a subgroup $G$ of $\Ham(X)$ isomorphic to
$(\ZZ/r)^n$, and an element $g\in G$ satisfying $X^{g}=X^{G}$. Since $X^g$ is nonempty,
so is $X^G$. Applying Lemma \ref{lemma:linearising-smooth} to each point in $X^G$,
and using the fact that if $m<2n$ then $\GL(m,\RR)$ has no subgroup isomorphic to
$(\ZZ/r)^n$, we conclude that all points in $X^G=X^g$ are isolated. If $x\in X^g$ then we denote
by $D_xg:T_xX\to T_xX$ the differential at $x$ of $g\in\Ham(X)$.
Since $g$ has finite order, all eigenvalues of $D_xg$ are roots of unity,
and since $x$ is isolated
in $X^g$, none of the eigenvalues of $D_xg$ is equal to $1$. This implies that $g$ is
a non degenerate Hamiltonian diffeomorphism. Finally, if $\alpha\in\CC\setminus\RR$ is an eigenvalue of $D_xg$, then so is $\ov{\alpha}$. All this implies that $\det(1-D_xg)>0$. By
Lefschetz's fixed point theorem, it follows that $\chi(X)=|X^g|$.
Applying the integral Arnold conjecture to $g$ we conclude that, for every prime $p$,
$$\dim_{\ZZ/p}H^*(X;\ZZ/p)\leq|X^g|=\chi(X)=\sum_k (-1)^k\dim_{\ZZ/p}H^k(X;\ZZ/p).$$ This implies that $H^k(X;\ZZ/p)=0$ for odd $k$ and that $\dim_{\ZZ/p}H^*(X;\ZZ/p)=\chi(X)$, which is independent of $p$. The theorem now follows from
the universal coefficients theorem and the fact that $H^*(X;\ZZ)$ is finitely generated.
\end{proof}

To conclude this section, let us mention that
the analogues of Question \ref{quest:abelian-2} and Question \ref{quest:abelian-2-2}
for birational transformation groups have been
answered in the affirmative by Aleksei Golota in \cite{Golota}.

\section{Free actions of finite abelian groups}
\label{s:free-abelian}

For any closed manifold $X$ we define the {\it discrete degree of free symmetry},
which we denote by $\discsym^{\free}(X)$, following the same recipe as for
the discrete degree of symmetry but considering only free actions of finite
abelian groups. Namely, we let $\mu^{\free}(X)$ denote the set of natural numbers
$m$ such that $X$ supports a free action of $(\ZZ/r)^m$ for arbitrary large integers
$r$, and we define
$$\discsym^{\free}(X):=(\max\{0\}\cup\mu^{\free}(X)).$$
Of course one always has $\discsym^{\free}(X)\leq\discsym(X)$.
The following theorem can be seen as a refinement of Mann--Su's theorem for
free actions. It was originally proved by Gunnar Carlsson \cite{Carlsson}
for $p=2$ and by Christoph Baumgartner \cite{Baumgartner}
for odd $p$ (see also \cite[Theorem 1.4.14]{AP}).

\begin{theorem}
\label{thm:Carlsson-Baumgartner}
Let $p$ be a prime and let $X$ be a paracompact topological space on which
$(\ZZ/p)^m$ acts freely and trivially on $H^*(X;\ZZ/p)$.
Suppose there exists some $i_0\in\NN$ such that $H^i(X/G;\ZZ/p)=0$
for all $i\geq i_0$. Then $m\leq |\{j\mid H^j(X;\ZZ/p)\neq 0\}|$.
\end{theorem}

The following result answers affirmatively Question \ref{quest:abelian-2}
for free actions. It is a consequence of Theorem \ref{thm:Carlsson-Baumgartner}
and Theorem \ref{thm:Minkowski} below.

\begin{theorem}
\label{thm:abelian-free-actions}
For any connected manifold $X$ we have $\discsym^{\free}(X)\leq \dim X$.
\end{theorem}

In \cite{Mundet2023-1} we prove the following, which gives a partial
affirmative answer to
Question \ref{quest:abelian-2-2} for free actions.

\begin{theorem}
\label{thm:abelian-free-actions-tori}
Let $X$ be a connected manifold satisfying $\discsym^{\free}(X)=\dim X$.
Then $H^*(X;\ZZ)\simeq H^*(T^n;\ZZ)$ as rings and the universal abelian cover of $X$
is acyclic. If in addition $\pi_1(X)$
is virtually solvable, then $X\cong T^n$.
\end{theorem}

Similarly, one can define the {\it stable free rank} of $X$, which we denote by
$\stablefreerank(X)$, by considering
free actions of $(\ZZ/p)^m$ for arbitrary large integers $p$.
This is trivially related to the discrete degree of symmetry
by the inequality $\stablefreerank(X)\leq\discsym^{\free}(X)$.
Ten years ago Bernhard Hanke proved the following remarkable result.

\begin{theorem}[Theorem 1.3 in \cite{Hanke}]
Let $X$ be a product of spheres $S^{j_1}\times\dots\times S^{j_k}$, and let
$k_o$ be the number of $j_i$'s which are odd. Then $\stablefreerank(X)=k_o$.
\end{theorem}

Actually Hanke proves more: if $(\ZZ/p)^r$ acts freely on the manifold $X$ of the theorem
and $p>3\dim X$, then $r\leq k_o$. Hanke's results naturally suggest the following question.

\begin{question}
Let $X$ be a product of spheres. What are the values of $\discsym(X)$,
$\stablerank(X)$ and $\discsym^{\free}(X)$?
\end{question}

In general we should expect $\discsym(X)$ (resp. $\stablerank(X)$)
to be bigger than $\discsym^{\free}(X)$ (resp. $\stablefreerank(X)$),
as for example $\discsym(S^n)=[(n+1)/2]$ (see Example \ref{ex:computation-discsym})
while $\discsym^{\free}(X)=1$ (this follows from Smith theory \cite{Smith}),
but in some cases this is not true. For example, if $X$ is a torus
(and perhaps also if $X$ is a product of spheres of dimensions $\leq 2$)
then $\discsym(X)=\discsym^{\free}(X)$.

\section{Nilpotent groups and beyond}
\label{s:nilpotent}

For a closed $n$-manifold $X$ with Jordan homeomorphism group, a positive answer to
Question \ref{quest:abelian-1} implies the existence of a constant $C$ such that
any finite group $G$ acting effectively on $X$ has an abelian subgroup $A\leq G$
satisfying $[G:A]\leq C$ and $d(A)\leq n$.
Since not all closed manifolds have Jordan homeomorphism
group, and in view of Theorem \ref{thm:CPS2}, it is natural to find analogues for
finite nilpotent groups of the questions and results in the previous sections.
We have the following result.

\begin{theorem}
\label{thm:generators-nilpotent}
Let $X$ be a closed and connected $n$-manifold. There exists a constant $C$
such that any finite nilpotent group $N$ acting effectively on $X$ has a subgroup
$N'$ satisfying $[N:N']\leq C$ and
$d(N')\leq (45n^2+6n+8)/8$.
\end{theorem}

Combined with Theorem \ref{thm:CPS2}, the previous theorem implies the following.

\begin{corollary}
Let $X$ be a closed and connected $n$-manifold. There exists a constant $C$
such that any finite group $G$ acting effectively on $X$ has a nilpotent subgroup
$N$ satisfying $[G:N]\leq C$ and $d(N)\leq (45n^2+6n+8)/8$.
\end{corollary}

We will deduce Theorem \ref{thm:generators-nilpotent} from previous results on actions of finite abelian groups
combined with a group theoretical result which we now state.
Given natural numbers $k,C$ let us denote by $\nN_{k,C}$
the collection of
all finite nilpotent groups $N$ such that every abelian subgroup
$A\leq N$ has a subgroup $B\leq A$ satisfying
$[A:B]\leq C$ and $d(B)\leq k$. The following theorem will be proved in Section
\ref{s:proof-thm:nilpotent-generators}.

\begin{theorem}
\label{thm:nilpotent-generators}
Given natural numbers $k,C$ there exists a constant
$C'=C'(k,C)$ such that every $N\in\nN_{k,C}$ has a subgroup $N'\leq N$
satisfying $[N:N']\leq C'$ and $d(N')\leq 1+k(5k+1)/2$.
\end{theorem}

Theorem \ref{thm:generators-nilpotent} follows from combining Theorem \ref{thm:weak-bound-disc-sym},
\cite[Lemma 2.7]{Mundet2021} and Theorem \ref{thm:nilpotent-generators}.
The bound on $d(N')$ given by Theorem \ref{thm:generators-nilpotent} is probably
far from optimal. Very probably, the bound resulting from combining
Theorem \ref{thm:nilpotent-generators} with a hypothetical
positive answer to Question \ref{quest:abelian-1} would neither be
optimal. We are thus led to the following question.

\begin{question}
\label{quest:nilpotent-generators}
Given a natural number $n$, what is the smallest number $\delta(n)$
with the property that for every closed and connected $n$-manifold $X$
there exists a constant $C$, depending only on $X$,
such that any finite nilpotent group $N$ acting effectively on $X$ has a subgroup
$N'$ satisfying $[N:N']\leq C$ and
$d(N')\leq \delta(n)$?
\end{question}

In \cite{Mundet2023-0} we prove:

\begin{theorem}
Let $X$ be an $n$-dimensional WLS manifold. There exists a number $p_0$ such that, for every
prime $p>p_0$, any finite $p$-group acting effectively on $X$ is nilpotent of nilpotency class
$2$ and can be generated by $n$ or fewer elements.
\end{theorem}

The class of WLS manifolds contains plenty of closed manifolds with non-Jordan
homeomorphism groups, but it is otherwise very small when compared with the collection of all closed manifolds. Nevertheless, the previous theorem suggests that asking whether the function $\delta(n)$ in Question \ref{quest:nilpotent-generators} is linear on $n$ is not completely crazy.

To conclude our discussion about the function $\delta$, let us mention the following result (see the paragraph after \cite[Corollary 1.3]{Mundet2016}), which can be seen as an extension of the theorem of Mann and Su to arbitrary finite groups:

\begin{theorem}
For any closed manifold $X$ there is a constant $C$ such that for
any finite group $G$ acting effectively on $X$ we have $d(G)\leq C$.
\end{theorem}

Besides the minimal number of generators, another natural invariant of finite nipotent
groups is the nilpotency class. We may ask the following.

\begin{question}
\label{quest:kappa}
Does there exist a function $\kappa:\NN\to\NN$ such that for any closed and connected manifold $X$ there is a constant $C$ with the property that any finite nilpotent group $N$ acting effectively on $X$ has a subgroup $N'\leq N$ of nilpotency class
$\kappa(\dim X)$ satisfying $[N:N']\leq C$?
If $\kappa(n)$ exists, what is its minimal possible value?
\end{question}

By (1) in Theorem \ref{thm:Jordan-homeo} the function $\kappa$ is defined on $\{1,2,3\}$
and satisfies $\kappa(1)=\kappa(2)=\kappa(3)=1$.
The main result in \cite{MundetSaez} is:

\begin{theorem}
Let $X$ be  closed $4$-dimensional smooth manifold. There exists a constant $C$ such that
every finite group $G$ acting smoothly and effectively on $X$ has a nilpotent subgroup
$N\leq G$ of nilpotency class $2$ and satisfying $[G:N]\leq C$.
\end{theorem}

Quite likely the arguments in \cite{MundetSaez} can be applied with some mild modifications to continuous actions on (topological) manifolds. This would imply that $\kappa(4)$ is well defined
and, taking into account the fact that $\Homeo(S^2\times T^2)$ is not Jordan, its value is
$\kappa(4)=2$.

Perhaps the strongest argument in favor of the existence of the function $\kappa$
comes from the analogy with birational transformation groups. Indeed, combining the works
of Golota \cite{Golota} and Guld \cite{Guld} we obtain the following.

\begin{theorem}
Let $X$ be a complex projective variety of complex dimension $n$.
There exists a constant $C$ such that any finite group $G$ of
birational transformations of $X$
has a nilpotent subgroup $G'$ of nilpotency class at most $2$ satisfying $[G:G']\leq C$ and $d(G')\leq 2n$.
\end{theorem}

More concretely, the previous theorem follows from applying first of all Guld's theorem
to reduce to the case of finite nilpotent groups of class $2$, and then applying Golota's
theorem to the action on the base and on the fiber over a generic point
of the maximal rationally connected fibration (see \cite[\S 2.3]{Guld}
and the proof of \cite[Theorem 23]{Guld}).

Could the function $\kappa$ in Question \ref{quest:kappa} be taken
to be constant equal to $2$ as for birational transformation groups? We don't have any argument against this possibility,
and in fact we do not know the answer to the following question.

\begin{question}
Does there exist a closed manifold $M$ supporting, for arbitrarily big primes $p$,
an effective action of a finite nilpotent $p$-group of nilpotency class $\geq 3$?
\end{question}

\section{Almost asymmetric manifolds}
\label{s:almost-asymmetric}
In the context of finite transformation groups, a manifold is said to be {\it asymmetric}
if it supports no effective action of a nontrivial finite group.
Closed asymmetric manifolds has been studied in a number of papers,
beginning with the examples of P.E. Conner, F. Raymond, P. Weinberger
\cite{CRW} and E.M. Bloomberg \cite{Bloomberg}. It is expected that in an
appropriate sense most manifolds are asymmetric (see \cite{Puppe} and the references therein). No example is presently known of a simply connected closed asymmetric
manifold, although we are probably close to it, see \cite[Theorem 4]{Puppe} and
\cite{Kreck}.

The point of view adopted in this paper suggest this definition: a manifold is {\it almost
asymmetric} if there is an upper bound on the size of the finite groups that act effectively on
it.

\begin{lemma}
\label{lemma:almost-asymmetric-disc-sym-zero}
If $X$ is a compact manifold then $X$ is almost asymmetric if and only if
$\discsym(X)=0$.
\end{lemma}
\begin{proof}
The "only if" part is obvious, so we only need to prove the "if" part.
Suppose that $\discsym(X)=0$. By \cite[Lemma 2.7]{Mundet2021}, there is a natural number $C$
such that any finite abelian group $A$ acting effectively on $X$ satisfies $|A|\leq C$.
Let $P$ be a finite $p$-group acting effectively on $X$. Let $A\leq P$ be a maximal
normal abelian subgroup. Conjugation of $P$ on $A$ induces a morphism $P/A\to\Aut(A)$
which is injective by \cite[\S 5.2.3]{Robinson}. Hence $|P|=|P/A|\cdot |A|\leq (C!)C$.
Let $G$ be a finite group acting effectively on $X$. For any prime $p$ dividing $|G|$
there is a cyclic (hence abelian) subgroup of $G$ of order $p$, so $p\leq C$.
Consequently there are at most $C$ primes dividing $|G|$. Applying the previous bound
to each of the Sylow $p$-subgroups of $G$ we deduce that $|G|\leq (C!)^CC^C$.
\end{proof}

Proving that a manifold is almost asymmetric is in general much simpler than proving that it is asymmetric.
The following result gives an example of this.

\begin{theorem}
\label{thm:torus-mes-torus}
For every $n\geq 2$ the manifold $T^n\sharp T^n$ is almost symmetric.
\end{theorem}
\begin{proof}
When $n$ is even $T^n\sharp T^n$ is a hypertoral manifold with nonzero Euler
characteristic, so for smooth actions the result follows from
\cite[Theorem 1.4 (2)]{Mundet2010}; the case of continuous actions
follows from the arguments in \cite[Theorem 1.4 (2)]{Mundet2010} combined with \cite[Theorem 4.1]{Mundet2021}.
If $n$ is odd the previous arguments don't apply, but one can use instead Lemma
\ref{lemma:almost-asymmetric-disc-sym-zero} combined with item (4) in
Example \ref{ex:computation-discsym} (see Theorem \ref{thm:suma-connexa-torus} below for more details).
\end{proof}

As another example, the manifolds in \cite[Theorem 4]{Puppe} are instances of simply
connected closed manifolds which are almost asymmetric. This makes it reasonable to
expect that the following vague question might be substantially more
accessible than the one addressed in \cite{Puppe}.

\begin{question}
Are "most closed manifolds" almost asymmetric?
\end{question}

For example, \cite[Theorem 6]{Puppe}, combined with Theorem \ref{thm:Minkowski} below,
implies that the answer to the previous question is affirmative when restricted
to the set $\nN$ of simply connected
spin 6-manifolds with free integral cohomology. To define "most manifolds"
\cite[Theorem 6]{Puppe} relies on the ring structure on the cohomology, which is
parametrized essentially by a degree three homogeneous polynomial with integer
coefficients on as many variables
as the rank of $H^2$. Consider, for each $n\in\NN$, the percentage
of all such polynomials that have coefficients in $[-n,n]$ and which
come from the cohomology of a manifold in $\nN$ which is not almost symmetric; then
\cite[Theorem 6]{Puppe} states that the limsup as $n\to\infty$ of this percentage
is equal to $0$.

\section{Trivial actions on homology}
\label{s:trivial-homology}
The following is a classical result of Hermann Minkowski \cite{Minkowski}.

\begin{lemma}
\label{lemma:Minkowski}
For each natural number $k$ there exists a number $C_k$ such that every
finite subgroup $G<\GL(k,\ZZ)$ satisfies $|G|\leq C_k$.
\end{lemma}

To prove the lemma it suffices to check that if
$\rho:\GL(k,\ZZ)\to\GL(k,\ZZ/3)$ is the componentwise reduction mod $3$,
and $a\in\GL(k,\ZZ)$ satisfies $\rho(a)=\rho(\Id)$ and $a^r=\Id$ for some
natural number $r$, then $a=\Id$. This implies that if $G<\GL(k,\ZZ)$ is finite
then $\rho|_G$ is injective, so $|G|\leq 3^{k^2}$.
Minkowski's lemma has the following implication.

\begin{theorem}
\label{thm:Minkowski}
Let $X$ be a compact manifold. There exists a constant $C$ such that, for
every action on $X$ of a finite group $G$, there is a subgroup $G'\leq G$
satisfying $[G:G']\leq C$ whose action on $H^*(X;\ZZ)$ is trivial.
\end{theorem}

The previous theorem is an immediate consequence of Lemma \ref{lemma:Minkowski}
when the cohomology of $X$ has no torsion. The case where there is some torsion
(which in any case will be finitely generated because $X$ is compact)
needs an easy extra argument, see \cite[Lemma 2.6]{Mundet2016}.

If $X$ is a noncompact manifold then a priori its cohomology might fail to
be finitely generated, so Lemma \ref{lemma:Minkowski} does not allow us to
conclude anything similar to Theorem \ref{thm:Minkowski} for actions on $X$.
However, there are some situations where $X$ is noncompact and one has
a statement similar to Theorem \ref{thm:Minkowski}. The following is proved
in \cite{Mundet2023-0}.

\begin{theorem}
\label{thm:Minkowski-equivariant}
Let $X$ be a manifold endowed with a properly discontinuous\footnote{An action of a discrete
group $G$ on a manifold $X$ is properly discontinuous if every $x\in X$ has a neighborhood
$U$ such that $U\cap g\cdot U=\emptyset$ for every $g\in G\setminus\{1\}$. This implies
that $X/G$ is a manifold.} action of $\ZZ^r$ such that
$X/\ZZ^r$ is a compact manifold. There exists a constant $C$ such that,
for every action on $X$ of a finite group $G$ that commutes with the action of $\ZZ^r$,
there is a subgroup $G'\leq G$ satisfying $[G:G']\leq C$ whose action on $H^*(X;\ZZ)$ is trivial.
\end{theorem}

In the previous theorem, the action of $\ZZ^r$ on $X$ endows $H^*(X;\ZZ)$ with a structure
of module over the group ring $\ZZ[\ZZ^r]\simeq\ZZ[t_1^{\pm 1},\dots,t_r^{\pm 1}]$. The assumption
that $X/\ZZ^r$ is compact implies that $H^*(X;\ZZ)$ is a finitely generated $\ZZ[\ZZ^r]$-module
(see \cite[Lemma 8.2]{Mundet2021}),
and Theorem \ref{thm:Minkowski-equivariant} follows from an analogue of Lemma
\ref{lemma:Minkowski} that is valid for finitely generated $\ZZ[\ZZ^r]$-modules. See
\cite{Mundet2023-0} for details.

\section{The rotation morphism}
\label{s:rotation}

In this section we explain how one can associate to the action of a finite group $G$
on a connected manifold $X$ and a $G$-invariant class $\alpha\in H^1(X;\ZZ)$ a character
$G\to S^1$. This construction is particularly useful in the case of rationally hypertoral
manifolds, especially when combined with Theorem \ref{thm:Minkowski}. For more details,
see \cite[\S 4]{Mundet2021}. In the case of smooth actions what we explain here can be
alternatively described using differential forms (see \cite[\S 2.1]{Mundet2010} and
\cite[\S 8.1]{MundetSaez}).

Here and everywhere we identify the circle $S^1$ with $\RR/\ZZ$ and use accordingly
additive notation for the group structure on $S^1$.

Let $\theta\in H^1(S^1;\ZZ)$ be a generator.
There exists a continuous map $\psi_{\alpha}:X\to S^1$, unique up to homotopy, such that $\psi_{\alpha}^*\theta=\alpha$. For each $g\in G$ we denote by $\rho(g):X\to X$
the map $x\mapsto g\cdot x$. Let
$\wt{\phi}_{\alpha}:=\sum_{g\in G}\psi_{\alpha}\circ\rho(g)$.
By construction we have $\wt{\phi}_{\alpha}\circ\rho(g)=\wt{\phi}_{\alpha}$
for every $g\in G$, i.e., $\wt{\phi}_{\alpha}$ is $G$-invariant.
Since  $\rho(g)^*\alpha=\alpha$ for every $g$, we have
$\wt{\phi}_{\alpha}^*\theta=|G|\alpha$. The fact that
$\wt{\phi}_{\alpha}^*\theta\in H^1(X;\ZZ)$ is $|G|$ times an integral class implies the
existence of a map $\phi_{\alpha}:X\to S^1$ such that $\wt{\phi}_{\alpha}=|G|\phi_{\alpha}$. The map $\phi_{\alpha}$ is not unique,
but two different choices of $\phi_{\alpha}$ differ by a constant
map $X\to S^1$ equal to some $|G|$-th root
of unity (i.e., the class in $\RR/\ZZ$ of an element of $|G|^{-1}\ZZ$).
We say that $\phi_{\alpha}$ is a $G$-th root of $\wt{\phi}_{\alpha}$.

For each $g\in G$ we have $|G|(\phi_{\alpha}\circ\rho(g))=(|G|\phi_{\alpha})\circ\rho(g)
=\wt{\phi}_{\alpha}\circ\rho(g)=\wt{\phi}_{\alpha}$. Hence,
$\phi_{\alpha}\circ\rho(g)$ is a $|G|$-th root of
$\wt{\phi}_{\alpha}$. Consequently, there is a
$|G|$-th root\footnote{The use of the word "root" here, as well as in the preceeding
sentence, is not consistent with the additive notation
for elements in $S^1$. A more appropriate word would be "fraction". However, since the terminology
"root of unity" is so universally used, we prefer to stick to it, fearing that the expression
"fraction of unity" would create more confusion than clarity.} of unity, $\xi_{\alpha}(g)\in S^1$, such that
$\phi_{\alpha}\circ\rho(g)=\xi_{\alpha}(g)+\phi_{\alpha}$.
This formula implies that the map $\xi_{\alpha}:G\to S^1$
is a morphism of groups. We call it the {\it rotation morphism}.
The morphism $\xi_{\alpha}$ is independent of the choice of
$\phi_{\alpha}$ and of the initial map $\psi_{\alpha}$: it only depends
on $\alpha$ and $\theta$. (Any two choices of $\psi_{\alpha}$
are homotopic, and $\xi_{\alpha}$ varies continuously with the choice
of $\psi_{\alpha}$ and takes values in a discrete set; hence $\xi_{\alpha}$
remains constant through any homotopy of maps $\psi_\alpha:X\to S^1$.)
Furthermore, $\xi_{\alpha}$ is linear on $\alpha$, as the reader can
easily check.

The following lemma will be used later when studying group actions on
nonorientable manifolds.

\begin{lemma}
\label{lemma:non-trivial-involution}
With notation as above, suppose that $\sigma:X\to X$ is an involution
satisfying $\sigma^*\alpha=-\alpha$, and suppose that the action of $G$
on $X$ commutes with $\sigma$. Then $2\xi_{\alpha}(g)=0$ for every $g\in G$.
\end{lemma}
\begin{proof}
The map $\psi_{-\alpha}:=\psi_{\alpha}\circ\sigma:X\to S^1$ satisfies
$\psi_{-\alpha}^*\theta=-\alpha$.
Let $\phi_{\alpha}$ be a $|G|$-th root of $\wt{\phi}_{\alpha}$. Then
$\phi_{\alpha}\circ\sigma$ is a $|G|$-th root of $\sum_{g\in G}\psi_{-\alpha}\circ\rho(g)$.
For any $x\in X$ we have
\begin{align*}
\xi_{\alpha}(x) &=\phi_{\alpha}(g\cdot x)-\phi_{\alpha}(x)=
\phi_{\alpha}(g\cdot \sigma(x))-\phi_{\alpha}(\sigma(x)) \\
&=\phi_{\alpha}(\sigma(g\cdot x))-\phi_{\alpha}(\sigma(x))
=(\phi_{\alpha}\circ\sigma)(g\cdot x)-(\phi_{\alpha}\circ\sigma)(x) \\
&=\xi_{-\alpha}(x)=-\xi_{\alpha}(x),
\end{align*}
which implies the lemma.
\end{proof}

More generally, if $\alpha_1,\dots,\alpha_k\in H^1(X;\ZZ)$ are $G$-invariant
classes, we denote $A=(\alpha_1,\dots,\alpha_k)$ and repeating the previous construction
for each $\alpha_i$ we define a map $\phi_A:X\to T^k=(S^1)^k$ and a morphism
$\xi_A:G\to T^k$ by
$\phi_A=(\phi_{\alpha_1},\dots,\phi_{\alpha_k})$ and
$\xi_A=(\xi_{\alpha_1},\dots,\xi_{\alpha_k})$.
By construction we have
$$\phi_A(g\cdot x)=\phi_A(x)+\xi_A(g).$$

We identify $T^k=(\RR/\ZZ)^k$ with $\RR^k/\ZZ^k$, and we denote by
$\pi_A:X_A\to X$ the pullback of the principal $\ZZ^k$-bundle
$\RR^k\to\RR^k/\ZZ^k$.

For each $g\in G$, $\rho(g):X\to X$ lifts to
a homeomorphism $\rho_A(g):X_A\to X_A$ that commutes with the
action of $\ZZ^k$ on $X_A$, and two choices of $\rho_A(g)$
differ by the action of an element of $\ZZ^k$ on $X_A$.
In general $\rho_A(g)$ will have infinite order, but not always.
In fact, we have:

\begin{lemma}
Let $g\in G$. There exists a finite order lift $\rho_A(g):X_A\to X_A$ of
$\rho(g)$ if and only if $\xi_A(g)=0$.
\end{lemma}

The following result, which is a restatement of \cite[Theorem 4.1]{Mundet2021},
shows the relevance of the rotation morphism for actions on rationally hypertoral manifolds.
It is the key ingredient in the proof that if $X$ is rationally hypertoral then $\Homeo(X)$
is Jordan.

\begin{theorem}
\label{thm:rotation-hypertoral}
Let $X$ be a closed connected and orientable $n$-manifold. Suppose that
$\alpha_1,\dots,\alpha_n\in H^1(X;\ZZ)$
satisfy $\alpha_1\smallsmile\dots\smallsmile\alpha_n=d\theta_X$, where $d$ is a nonzero integer
and $\theta_X\in H^n(X;\ZZ)$ is a generator.
Let $A=(\alpha_1,\dots,\alpha_n)$.
For any action of a finite group $G$ on $X$ inducing the trivial
action $H^1(X;\ZZ)$ the morphism $\xi_A:G\to T^n$ satisfies $|\Ker\xi_A|\leq |d|$.
\end{theorem}

Keeping with the previous notation, the action of $\ZZ^k$ on $X_A$
induces an action on $H^*(X_A;\ZZ)$ which allows us to look at
$H^*(X_A;\ZZ)$ as a module over the group ring $\ZZ[\ZZ^k]$.
If $X$ is closed, then $H^*(X_A;\ZZ)$ is finitely generated as a
$\ZZ[\ZZ^k]$-module by \cite[Lemma 8.2]{Mundet2021}. In general it is not finitely generated as a
$\ZZ$-module, but the following result can be used to show that it is so under some
conditions (see the proof of Theorem \ref{thm:suma-connexa-torus} below).
The next theorem is \cite[Corollary 6.3]{Mundet2021}.

\begin{theorem}
\label{thm:finite-generation}
Let $B=\ZZ[z_1^{\pm 1},\dots,z_b^{\pm 1}]$.
Let $M$ a finitely generated $B$-module. Suppose there exists
some nonzero $\delta\in\NN$, and for every $1\leq j\leq b$
integers $r_{j,i}\to\infty$ as $i\to\infty$, and $w_{j,i}\in\End_BM$
such that $w_{j,i}^{r_{j,i}}=z_j^{\delta}$ for each pair $i,j$.
Then $M$ is a finitely generated $\ZZ$-module.
\end{theorem}

The previous theorem is a key ingredient in the proof of Theorem
\ref{thm:disc-sym-rationally-hypertoral-WLS}.
If $X$ is a closed connected $n$-manifold and
$\alpha_1,\dots,\alpha_n\in H^1(X;\ZZ)$
satisfy $\alpha_1\smallsmile\dots\smallsmile\alpha_n\neq 0$,
then the combination of Theorems \ref{thm:Minkowski-equivariant} and \ref{thm:rotation-hypertoral} with Theorem \ref{thm:finite-generation},
applied to $M=H^*(X_A;\ZZ)$ with $A=(\alpha_1,\dots,\alpha_n)$
and $b=n$, implies that $H^*(X_A;\ZZ)$ is finitely generated as a
$\ZZ$-module. Then, an argument based on Serre's spectral
sequence implies
that $X_A$ must actually be acyclic, which implies that $H^*(X;\ZZ)\simeq H^*(T^n;\ZZ)$. If $\pi_1(X)$ is solvable then $X_A$ is contractible,
so $X$ is homotopy equivalent to $T^n$ and hence it is also homeomorphic to
it (by the topological rigidity of tori). If $\pi_1(X)$ is virtually solvable
then the argument proceeds by passing to a suitable finite cover of $X$.
The case of WLS manifolds follows a similar
strategy, but there are some additional difficulties. An important ingredient
in that case is Theorem \ref{thm:weak-fixed-point} below.

The constructions in this section are also used in the proof of statement(4) in
Example \ref{ex:computation-discsym} (see \cite[Theorem 1.5]{Mundet2021}).
Let us illustrate this in a particular case, which can be used to complete
the proof of Theorem \ref{thm:torus-mes-torus} above.

\begin{theorem}
\label{thm:suma-connexa-torus}
Let $X$ be a closed and connected $n$-manifold satisfying $H^*(X;\ZZ)\not\simeq H^*(S^n;\ZZ)$.
Then $\discsym(T^n\sharp X)=0$.
\end{theorem}
\begin{proof}
Suppose that there exists a sequence of positive integers $m_i\to\infty$ and, for each $i$,
an effective action of $G_i:=\ZZ/m_i$ on $Y:=T^n\sharp X$.
We will see that this leads to a contradiction.
Replacing each $G_i$ by an appropriate
subgroup of itself we can assume that $|G_i|$ is a prime power for each $i$, while still having $r_i:=|G_i|\to\infty$.
By Theorem \ref{thm:Minkowski} we may and will assume that the induced action of $G_i$ on $H^*(Y;\ZZ)$
is trivial.

Choose $a_1,\dots,a_n\in H^1(T^n;\ZZ)$ such that $a_1\smallsmile\dots\smallsmile a_n$ generates $H^n(T^n;\ZZ)$,
and let $\alpha_i\in H^1(Y;\ZZ)$ be the pullback of $a_i$ through the projection $q:Y\to T^n$
arising from collapsing the summand $X$ in $Y$.
Then $\alpha_1\smallsmile\dots\smallsmile \alpha_n$ generates $H^n(Y;\ZZ)$.

Let $\xi_i:G_i\to T^n$ be the rotation morphism associated to $A=(\alpha_1,\dots,\alpha_n)$
and the action of $G_i$ on $Y$, and let $\phi_i:Y\to T^n$ the corresponding $G_i$-equivariant
map. Note that the maps $\phi_i$ are pairwise homotopic.
By Theorem \ref{thm:rotation-hypertoral}, $\xi_i$ is injective.
Choose a generator $g_i\in G_i$ and let $\xi_i(g_i)=(\theta_{i1},\dots,\theta_{in})$. Writing
$r_i=p_i^{e_i}$, where $p_i$ is prime and $e_i$ an integer, each $\theta_{ij}$ has order $p_i^{e_{ij}}$
for some integer $e_{ij}$. Since $\xi_i$ is injective, we have $e_i=\max_j e_{ij}$.
Passing to a subsequence, we may assume that for some $j$ we have $e_i=e_{ij}$. So, if
$\pi_j:T^n\to S^1$ is the projection to the $j$-th factor then $\rho_i:=\pi_j\circ\xi_i:G_i\to S^1$
is injective.

Let $Y_i\to Y$ the pullback of the fibration $\RR\to\RR/\ZZ=S^1$ through the
map $\zeta_i:=\pi_j\circ \phi_i:Y\to S^1$.
There is an action of $\ZZ$ on $Y_i$ induced by the action by translations of $\ZZ$ on $\RR$.
All maps $\zeta_i$ correspond to the same class $\alpha_j\in H^1(Y;\ZZ)$, hence are pairwise homotopic.
Consequently we can identify $\ZZ$-equivariantly $Y_i\simeq Y_1$ for each $i$. The action of $\ZZ$ on $H^*(Y_1;\ZZ)$
induces a structure of $\ZZ[z,z^{-1}]$-module on $H^*(Y_1;\ZZ)$, where multiplication by $z$ corresponds
to the action of $1\in\ZZ$. Since $Y$ is compact,
$H^*(Y_1;\ZZ)$ is finitely generated as a $\ZZ[z,z^{-1}]$-module by \cite[Lemma 8.2]{Mundet2021}.

Let $\gamma_i\in G_i$ be the element such that $\rho_i(\gamma_i)=[r_i^{-1}]\in\RR/\ZZ=S^1$.
There is a lift $\tau_i:Y_i\to Y_i$ of the action of $\gamma_i$ on $Y$
such that $\tau_i^{r_i}$ is equal to the action of $1\in\ZZ$ on $Y_i$.
Via the identification $Y_1\simeq Y_i$, $\tau_i$ corresponds to a homeomorphism
$\tau_i'$ of $Y_1$. Let $w_i:=(\tau_i')^*\in \Aut H^*(Y_1;\ZZ)$.
We have $w_i^{r_i}=z$. Since $r_i\to\infty$, Theorem \ref{thm:finite-generation} implies that $H^*(Y_1;\ZZ)$
is finitely generated as a $\ZZ$-module.

Let $\kappa:T_1\to T^n$ be the pullback of the fibration $\RR\to\RR/\ZZ=S^1$ via the map $T^n\to S^1$ given by
$a_j\in H^1(T^n;\ZZ)$. Then, for some $q\in T^n$, $Y_1$ can be identified with the connected sum of $T_1$ and
infinitely many copies of $X$, one at each point of $\kappa^{-1}(q)$.
Since $H^*(X;\ZZ)\not\simeq H^*(S^n;\ZZ)$, each time we make connected sum with $X$ some of the Betti numbers
increases, so if we perform connected sum with $X$ infinitely many times we obtain a manifold whose integral
cohomology is not finitely generated. This is a contradiction, so the proof of the theorem is complete.
\end{proof}

We close this section with a result that will be used in the proof of Theorem \ref{thm:disc-sym-low-dimension}.

\begin{lemma}
\label{lemma:free-actions-3-manifold}
Let $X$ be a closed $3$-manifold satisfying $H^*(X;\ZZ)\simeq H^*(T^3;\ZZ)$. Any effective
action of $S^1$ on $X$ is free.
\end{lemma}
\begin{proof}
Suppose given an effective action of $S^1$ on $X$ which is not free.
Since $S^1$ is connected, the induced action of $S^1$ on $H^*(X;\ZZ)$ is trivial.
Let $x\in X$ be a point with nontrivial stabilizer, and let $g\in G_x$ be
a nontrivial element of finite order. Let $G<S^1$ be the subgroup
generated by $g$. Since $H^*(X;\ZZ)\simeq H^*(T^3;\ZZ)$, there are classes
$A=(\alpha_1,\alpha_2,\alpha_3)\in H^1(X;\ZZ)^3$ such that $\alpha_1\smallsmile\alpha_2\smallsmile\alpha_3$
is a generator of $H^3(X;\ZZ)$. By Theorem \ref{thm:rotation-hypertoral},
$\xi_A:G\to T^n$ is injective. But the equivariance property
$\phi_A(g\cdot x)=\phi_A(x)+\xi_A(g)$ implies that $\xi_A(g)=0$, because $g\cdot x=x$,
so we are led to a contradiction.
\end{proof}

The previous result has an obvious generalisation to arbitrary dimensions. One can prove in addition
that if $X$ is a closed $n$-manifold satisfying $H^*(X;\ZZ)\simeq H^*(T^n;\ZZ)$ then
any effective action of $S^1$ on $X$ is free and the orbits represent a nontrivial element in $H_1(X;\ZZ)$. This implies that $X\cong Y\times S^1$, where $Y=X/S^1$ (since $S^1$
acts freely on $X$, the quotient map $X\to Y$ is a circle bundle, and the nontriviality of the
orbits in $H_1(X;\ZZ)$ implies by Gysin that the Euler class of the circle bundle $X\to Y$
is trivial).

\section{Fixed points}
\label{s:fixed-points}

A basic question in topology is to find conditions under which a self map, or in
particular a self homeomorphism, of a given topological manifold has necessarily
a fixed point. Typical examples are Brouwer's or Lefschetz's fixed point theorems.
When considering an action of a finite group $G$
on a manifold $X$, these results may imply in some cases the
existence of a big subgroup $G'\leq G$ all of whose elements act on $X$
with a fixed point. But {\it a priori} there need not be
any relation between the fixed points of different elements in $G'$, let
alone a point of $X$ fixed by all elements of $G'$. For example, while for $n\leq 3$
any action of a finite group on the closed $n$-dimensional disk $D^n$ has a fixed point, there are
examples of actions of finite groups on $D^n$, for $n\geq 6$, without fixed points,
although Brouwer's fixed point theorem implies that each element of the group
fixes at least one point of the disk (see the Introduction in \cite{Mundet2020}).
However, and perhaps surprisingly, there is a subgroup, of index bounded by a function
of $n$, which does have a fixed point. This is a general property, as we will see.

Let us say that an action of a group $G$ on a space $X$ has the weak {\it fixed point property}
if for every $g\in G$ there is some $x\in X$ such that $g\cdot x=x$. The following
is \cite[Theorem 1.6]{Mundet2022}.

\begin{theorem}
\label{thm:weak-fixed-point}
Let $X$ be a connected manifold with $H_*(X;\ZZ)$ finitely generated. There exists a constant
$C$ with this property: given any action of a finite group $G$ on $X$ with
the weak fixed point property, there is a subgroup $G'\leq G$
satisfying $[G:G']\leq C$ and $X^G\neq\emptyset$.
\end{theorem}

Combining Theorem \ref{thm:weak-fixed-point} with Theorem \ref{thm:Minkowski}
and with Lefschetz's fixed point theorem we obtain the following.

\begin{theorem}
\label{thm:almost-fixed-point-nonzero-Euler}
Let $X$ be a compact manifold satisfying $\chi(X)\neq 0$. There is a constant
$C$ such that for any action of a finite group $G$ on $X$ there is a subgroup
$G'\leq G$ satisfying $[G:G']\leq C$ and $X^{G'}\neq\emptyset$.
\end{theorem}

The case $X=D^n$ is already nontrivial, and the examples of finite group actions
without fixed points when $n\geq 6$ show that the need to pass to a subgroup of
$G$ is in general unavoidable.

The existence of fixed points has strong implications regarding the algebraic
structure of the group that acts. This is most transparent in the smooth category,
in which we have the following result.

\begin{lemma}
\label{lemma:linearising-smooth}
Let $X$ be a connected smooth manifold, and suppose that a finite group $G$
acts smoothly and effectively on $X$. If $x\in X^G$ then the morphism $\delta:G\to\GL(T_xX)$,
defined by deriving the action of $G$ on $X$ at $x$, is injective.
\end{lemma}
\begin{proof}
Let $\eta_0$ be any Riemannian metric on $X$ and let $\eta=\sum_{g\in G}\rho(g)^*\eta_0$,
where $\rho(g)$ is the diffeomorphism $X\ni x\mapsto g\cdot x\in X$.
Then $\eta$ is a $G$-invariant Riemannian metric on $X$, so the exponential map
with respect to $\eta$ gives a $G$-equivariant diffeomorphism from a neighborhood
of $0$ in $T_xX$ to a neighborhood of $x$ in $X$. Therefore, if $g\in\Ker\delta$ then
$X^g=\{x\in X\mid g\cdot x\}$ has nonempty interior. The same argument applied to
the action of the subgroup $\la g\ra\leq G$ generated by $g$ implies that
the interior of $X^g=X^{\la g\ra}$ is closed. Since the interior $X^g$ is obviously open,
and since $X$ is connected, it follows that $X^g=\emptyset$, which implies that
$g=1$ because $G$ acts effectively on $X$.
\end{proof}

For example, if $X$ is a connected $n$-dimensional smooth manifold
and $G=(\ZZ/r)^m$ acts smoothly and effectively on $X$ with a fixed point
then $G$ is isomorphic to a subgroup of $\GL(n,\RR)$, which easily implies
that $m\leq [n/2]$ as soon as $r\geq 3$.

If instead of a smooth action we consider a continuous action of a finite group
$G$ on a connected $n$-dimensional manifold, then the existence of a fixed point does
not necessarily imply that $G$ is isomorphic to a subgroup of $\GL(n,\RR)$.
Indeed, Bruno Zimmermann has constructed in \cite{Zimmermann2017}, for each
$n\geq 5$, examples
of effective continuous actions on the $n$-sphere $S^n$ of groups which are not
isomorphic to any subgroup of $\GL(n+1,\RR)$. Taking the cone of $S^n$, which is
homeomorphic to $\RR^{n+1}$, we obtain an effective
action that fixes a point (namely, the vertex of the cone). However, if we consider
actions of finite $p$-groups then no such example exists:

\begin{lemma}
\label{lemma:linearising-continuous}
Let $X$ be a connected $n$-manifold, and suppose that a finite $p$-group $G$
acts continuously and effectively on $X$. If $X^G\neq\emptyset$ then $G$
is isomorphic to a subgroup of $\GL(n,\RR)$.
\end{lemma}

The previous lemma follows from combining a result of
R.M. Dotzel and G.C. Hamrick \cite{DH} with basic results on the geometry
of continuous actions of $p$-groups near a fixed point
(see \cite[Corollary 3.3]{Mundet2022} for details).

Lemma \ref{lemma:linearising-continuous} is one of the ingredients in the proof
that if $X$ is a connected manifold with $H_*(X;\ZZ)$ is finitely generated
and $\chi(X)\neq 0$ then $\Homeo(X)$ is Jordan. Let us briefly sketch the argument of the proof.
If $H_*(X;\ZZ)$ is finitely generated
and $\chi(X)\neq 0$ then a formula of Ye \cite[Theorem 2.5]{Ye2} implies
the existence of a number $C$ such that any finite $p$-group $G$ acting effectively on $X$
has a subgroup $G'\leq G$ satisfying $[G:G']\leq C$ and $X^{G'}\neq\emptyset$
(see \cite[Lemma 2.1]{Mundet2022} for details).
Lemma \ref{lemma:linearising-continuous} implies that $G'$ is isomorphic
to a subgroup of $\GL(n,\RR)$, where $n=\dim X$, and Jordan's theorem then implies
that $G'$ has an abelian subgroup $A\leq G'$ satisfying $[G':A]\leq C'$
for some $C'$ depending only on $n$ (and hence on $X$, but not on $G$).
This fact on $p$-groups can be combined with the result
by Csik\'os, Pyber and  Szab\'o mentioned after the statement of Theorem \ref{thm:CPS2}
above to conclude that $\Homeo(X)$ is Jordan.
Note that we are not assuming that $X$ is compact
as in Theorem \ref{thm:almost-fixed-point-nonzero-Euler}; but compactness is crucial in
Theorem \ref{thm:almost-fixed-point-nonzero-Euler}, as long as we care about arbitrary
finite groups and not only on $p$-groups, as shown for example by
the main result in \cite{HKMS} (see the end of \cite[\S 1.1]{Mundet2017} for an explanation).

\section{Stabilizers}
\label{s:stabilizers}

Studying whether a given action has fixed points is a particular case of the
problem of understanding the collection of stabilizers of the action.
For any effective action of a group $G$ on a space $X$ we denote by
$$\Stab(G,X)=\{G_x\mid x\in X\}$$
the collection of subgroups of $G$ that arise as stabilizers of points in $X$. The previous
section was concerned with the question of whether $G\in \Stab(G,X)$. Here we
consider the question of bounding the cardinal of $\Stab(G,X)$.

In general, $|\Stab(G,X)|$ cannot be bounded by a constant depending only $X$.
We prove this with an example. If $G_n$ denotes the isometry group of a
regular $n$-gon $P_n\subset\RR^2$, then two vertices of $P_n$ have the same stabilizer in $G_n$
if and only if they are aligned with the center of $P_n$, so $|\Stab(G_n,P_n)|\geq n/2$. Since
$P_n\cong S^1$ for every $n$, we obtain actions of finite groups $G$ on $S^1$ with arbitrarily
large $|\Stab(G,S^1)|$. Taking $n=2^k$ the group $G_n$ is a $2$-group, so this phenomenon also
holds true for $p$-groups (examples of $p$-group actions for an arbitrary prime $p$
with the same behaviour can be obtained
taking actions on the $p-1$-dimensional torus, see \cite[Remark 1.4]{CMPS}).

In contrast, \cite[Theorem 1.3]{CMPS} states the following.

\begin{theorem}
\label{thm:CMPS}
Let $X$ be a manifold with finitely generated $H_*(X;\ZZ)$. There exists a constant $C$
such that for every action of a finite $p$-group $G$ on $X$ there is a subgroup $H\leq G$
containing the center of $G$ and satisfying $[G:H]\leq C$ and $|\Stab(H,X)|\leq C$.
\end{theorem}

Theorem \ref{thm:CMPS} plays a key role in the proof of many of the results stated in this survey, as will become clear in the following three sections.

Is Theorem \ref{thm:CMPS} true if instead of $p$-groups we consider arbitrary finite groups? This
seems to be unknown at present, with the exception of $D^n$, $S^{2n}$ and, for trivial
reasons, almost asymmetric manifolds. The following is \cite[Question 1.9]{CMPS}:

\begin{question}
Does there exist, for every compact manifold $X$, a constant $C$ such that for any action
of a finite group $G$ on $X$ there is a subgroup $H\leq G$ satisfying $[G:H]\leq C$
and $|\Stab(H,X)|\leq C$?
\end{question}

\section{Proofs of some of the theorems}
\label{s:proofs}

\subsection{Proof of Theorem \ref{thm:smooth-disc-sym-low-dimension}}
\label{ss:thm:smooth-disc-sym-low-dimension}
The case $n=1$ is elementary. For the case $n=2$ see the comments in item (2) of
Example \ref{ex:computation-discsym}.
Let us now prove the case $n=4$, and afterwards we will address the case $n=3$.

We first prove some elementary facts on finite abelian groups to be used later.
Recall that for any finite group $G$ we denote by $d(G)$
the minimal size of a generating set of $G$.
If $p$ is a prime and $G$ is a finite abelian $p$-group of
exponent $p$, then we can look at $G$ as a vector space over $\ZZ/p$,
and a subset $g_1,\dots,g_k\in G$ is a minimal generating subset if
and only if it is a basis, so $d(G)=\dim G$. Therefore, $d((\ZZ/p)^m)=m$.

As elsewhere in this paper, we use additive notation for abelian groups.

\begin{lemma}
\label{lemma:d-non-decreasing}
If $A\leq B$ are finite abelian groups, then $d(A)\leq d(B)$.
\end{lemma}
\begin{proof}
Let $B$ be a finite abelian group.
If $\pi:B\to C$ is a surjection then $\pi$ maps any generating
set of $B$ to a generating set of $C$, so $d(C)\leq d(B)$.
By Pontryagin duality (see e.g. \cite[(A3),(A11)]{FrTa}),
given a finite abelian group $A$ there is a subgroup of $B$ isomorphic to $A$
if and only if there is a quotient of $B$ isomorphic to $A$.
The lemma follows.
\end{proof}

The previous result is false for arbitrary finite groups. For example, for any
natural number $n\geq 3$ the symmetric group $S_{2n}$ is generated by $(1\,2)$ and
$(1\,2\,\dots\, 2n)$,
which implies that $d(S_{2n})=2$, because $S_{2n}$ is not cyclic. But
the subgroup $G\leq S_{2n}$ generated by the transpositions
$(1\,2),(3\,4),\dots,(2n-1\,\,\,2n)$ is isomorphic to $(\ZZ/2)^n$, so $d(G)=n$.

\begin{lemma}
\label{lemma:d-abelian-quotient}
Let $A$ be a finite abelian $p$-group.
We have $d(A)=d(A/pA)$.
\end{lemma}
\begin{proof}
By the classification of finite abelian groups we have
$A\simeq\ZZ/p^{e_1}\times\dots\times\ZZ/p^{e_k}$ for some
naturals $k,e_1,\dots,e_k$. Then $A/pA\simeq(\ZZ/p)^k$.
Since $A$ can be generated by $k$ elements, we have $d(A)\leq d(A/pA)$.
The converse inequality $d(A)\geq d(A/pA)$ follows from the existence of a surjection $A\to A/pA$.
\end{proof}

The previous result is a particular case of the following general fact: if $G$ is a finite group and $F(G)$ denotes
its Frattini subgroup (i.e., the intersection of all maximal proper subgroups of $G$) then
$d(G)=d(G/F(G))$ (see e.g. \cite[Lemma X.5]{Isaacs}).

\begin{lemma}
\label{lemma:small-d-quotient-big-K}
Let $p$ be a prime and let $e,m$ be natural numbers.
Let $G=(\ZZ/p^e)^m$ and let $K\leq G$ be a subgroup. Let $n=d(G/K)$. We have
$n\leq m$ and if $n<m$ then $K$ contains a subgroup isomorphic to $(\ZZ/p^e)^{m-n}$.
\end{lemma}
\begin{proof}
The inequality $n\leq m$ follows from the fact that $G/K$ is a quotient of $G$.
By Lemma \ref{lemma:d-abelian-quotient} we have
$d(G/K)=d((G/K)/p(G/K))=d(G/(K+pG))$. There is a natural exact sequence
$$0\to (K+pG)/pG \to G/pG\to G/(K+pG)\to 0,$$
where each of the groups have exponent $p$. Hence
$$\dim((K+pG)/pG)=\dim (G/pG)-\dim(G/(K+pG))=m-n.$$
Let $\varepsilon_1,\dots,\varepsilon_{m-n}$
be a basis of $(K+pG)/pG$ as a $\ZZ/p$-vector space.
Choose for each $i$ a lift $e_i\in K+pG$ of $\varepsilon_i$.
Since $(K+pG)/pG\simeq K/(K\cap pG)$, we can in fact take $e_i$ inside $K$.
The morphism $E:\ZZ^{m-n}\to K$, $(l_1,\dots,l_{m-n})\mapsto
\sum l_ie_i$ has kernel $p^e\ZZ^{m-n}$. This follows from the general fact that if $f_1,\dots,f_r\in G$ are lifts of a set of $r$ linearly independent elements of $G/pG$ then the kernel of the morphism
$F:\ZZ^r\to G$, $(\lambda_1,\dots,\lambda_r)\mapsto\sum_i\lambda_if_i$, is equal to $p^e\ZZ^r$. Certainly $p^e\ZZ^r\leq\Ker F$. To prove the reverse inclusion $\Ker F\leq p^e\ZZ^r$, assume that $\sum_i p^{c_i}b_if_i=0$ where $c_i\geq 0$, $b_i$
are integers and $b_i$ is not divisible by $p$. Let $c=\min\{c_i\}$, and assume that $c<e$.
Then we have
$\sum_{i\mid c_i=c} p^cb_if_i\in p^{c+1}G$, which implies that
$\sum_{i\mid c_i=c} b_if_i\in pG$ because $c<e$.
This contradicts the fact that the projections of $f_1,\dots,f_r$
in $G/pG$ are linearly independent, and thus the proof that $p^e\ZZ^r=\Ker F$ is now complete.
It follows that $E(\ZZ^{m-n})\leq K$ is isomorphic to $\ZZ^{m-n}/p^e\ZZ^{m-n}\simeq (\ZZ/p^e)^{m-n}$.
\end{proof}

We are going to use the following result.

\begin{lemma}
\label{lemma:restriction-to-surface}
Let $\Sigma$ be a closed and connected surface of genus $g$. There is a constant integer
$C$, depending only on $g$, with the following properties.
\begin{enumerate}
\item If $p>C$ is a prime and $m\geq 2$ is an integer, then for any morphism $\phi:(\ZZ/p)^m\to\Diff(\Sigma)$
the kernel of $\phi$ contains a subgroup isomorphic to $(\ZZ/p)^{m-2}$.
\item If $p\leq C$ is a prime and $e\geq C$ is an integer, then
for any morphism $\phi:(\ZZ/p^e)^m\to\Diff(\Sigma)$
the kernel of $\phi$ contains a subgroup isomorphic to $(\ZZ/p^{e-C})^{m-2}$.
\end{enumerate}
\end{lemma}
\begin{proof}
Combining the comments in item (2) of
Example \ref{ex:computation-discsym} with the classification of closed connected surfaces
we deduce the existence of a constant $C\geq 1$ such that any finite abelian subgroup $A<\Diff(\Sigma)$
has a subgroup $A'\leq A$ satisfying $[A:A']\leq C$
and $d(A')\leq 2$. We claim that this constant $C$ satisfies the desired properties.

We first prove (1). Let $p>C$ be a prime, let $m\geq 2$ be an integer, and suppose that $\phi:G\to\Diff(\Sigma)$ is a group morphism, where $G:=(\ZZ/p)^m$. Any subgroup or quotient of $G$ can naturally be seen as a vector space
over $\ZZ/p$. There is a subgroup $A'\leq\phi(G)$ which satisfies $d(A')\leq 2$ and
$[\phi(G):A']\leq C$. Since $\phi(G)$ is a $p$-group and $p>C$, we have $A'=\phi(G)$. We have
$\dim\phi(G)=d(G)\leq 2$. It follows that $\dim\Ker\phi\geq\dim G-2=m-2$. Hence $\Ker\phi$ contains a vector subspace of dimension $m-2$,
which is hence isomorphic to $(\ZZ/p)^{m-2}$.

We now prove (2).
Let $p\leq C$ be a prime and let $e\geq 1$ and $m\geq 2$ be integers.
Let $G=(\ZZ/p^e)^m$ and let $\psi:G\to\Diff(\Sigma)$ be a morphism of groups.
There is a subgroup $A'\leq\psi(G)$ satisfying $d(A')\leq 2$ and
$[\psi(G):A']\leq C$. Let $G'=\psi^{-1}(G)$. Then $|G/G'|=p^s$ for some integer
$s$, and we have $p^s\leq C$, which implies that $s\leq\log_pC\leq C$.
Let $G''=p^sG$.
Then $G''\leq (\ZZ/p^{e-s})^m$ and $G''\leq G'$. The latter implies that $\psi(G'')\leq A'$, so $d(\psi(G''))\leq 2$ by Lemma \ref{lemma:d-non-decreasing}.
By Lemma \ref{lemma:small-d-quotient-big-K}, $\Ker\psi$ contains a subgroup
isomorphic to $(\ZZ/p^{e-s})^{m-2}$.
\end{proof}

We are now ready to prove Theorem \ref{thm:smooth-disc-sym-low-dimension}.
Suppose that $X$ is a smooth closed, connected $4$-manifold. Suppose also, for the time
being, that $X$ is orientable, and choose an
orientation of $X$. Suppose that $r_i\to\infty$ is a sequence of integers
$m$ is a natural number such that $G_i:=(\ZZ/r_i)^m$ acts smoothly and effectively
on $X$ for every $i$. Without loss of generality we assume that $G_i$ acts on $X$
preserving the orientation (apply \cite[Lemma 2.1]{Mundet2021} to the kernel of
$G_i\to\Aut(H^4(X;\ZZ))$). Also, as in the proof of Theorem \ref{thm:toric-finite-subgroup},
we may assume that either each $r_i$ is
a prime, or there exists some prime $p$ such that each $r_i$ is a power of $p$.

Passing to a subsequence if necessary, we can assume that
either the action of $(\ZZ/r_i)^m$ is free for every $i$, or that it has
nontrivial stabilizers for every $i$. In the first case, Theorem \ref{thm:abelian-free-actions}
implies that $m\leq 4$, and if $m=4$ then Theorem \ref{thm:abelian-free-actions-tori}
allows to conclude the proof of Theorem \ref{thm:smooth-disc-sym-low-dimension}.

Now suppose that for each $i$ there is some nontrivial subgroup $\Gamma_i\leq G_i$
such that $X^{\Gamma_i}\neq\emptyset$. We will see that in this case we necessarily have $m\leq 3$.
Choose for each $i$ a connected component $Y_i$ of
$X^{\Gamma_i}$. By Smith theory (see \cite[Chap III, Theorem 4.3]{Bo}), $|\pi_0(X^{\Gamma_i})|$
is bounded above by a constant depending only on $X$ (recall that $r_i$ is
a power of a prime), and hence using again \cite[Lemma 2.1]{Mundet2021} we may assume
that the action of $G_i$ on $X$ preserves $Y_i$ for each $i$.
Since $G_i$ acts on $X$ preserving the orientation,
$Y_i$ has even codimension in $X$, and hence it is $0$ or $2$-dimensional.
(If $p$ is odd then $Y_i$ has even codimension regardless of whether the action is orientation
preserving. It's only in the case $p=2$ that orientability plays a role.)
Passing to a subsequence we may assume that $\dim Y_i$ is independent of $i$.
If $\dim Y_i=0$ then $Y_i$ is a fixed point of $G_i$, so Lemma \ref{lemma:linearising-smooth}
gives an embedding
$G_i\hookrightarrow\GL(4,\RR)$, which implies that $m\leq 2$.
Suppose now that $\dim Y_i=2$ for every $i$. By \cite[Chap III, Theorem 4.3]{Bo}
and the theorem on classification of closed connected surfaces, the genus of $Y_i$ is bounded
uniformly. Hence we may apply Lemma \ref{lemma:restriction-to-surface} to the morphisms
$\phi_i:G_i\to\Diff(Y_i)$ with a constant $C$ independent of $i$. We distinguish two cases.

Suppose that each $r_i$ is prime. Then, by item (1) in Lemma \ref{lemma:restriction-to-surface},
the kernel of $\phi_i$ contains a subgroup $H_i\simeq(\ZZ/r_i)^{m-2}$. Let $y_i\in Y_i$ be
any point. Since $H_i$ is contained in the kernel of $\phi_i$, all elements of $H_i$ fix
$y_i$, and consequently by Lemma \ref{lemma:linearising-smooth} there is an
embedding\footnote{Lemma \ref{lemma:linearising-smooth} states that
the map $\lambda_i:H_i\hookrightarrow\GL(T_{y_i}X)$ given by linearising the action is injective;
but the image of $\lambda_i$ is contained in the group $\GL(T_{y_i}X,T_{y_i}Y_i)$
of automorphisms of $T_{y_i}X$
acting trivially on $T_{y_i}Y_i\subset T_{y_i}X$. Since $H_i$ is finite, composing
$H_i\hookrightarrow \GL(T_{y_i}X,T_{y_i}Y_i)$ with the projection $\GL(T_{y_i}X,T_{y_i}Y_i)\to\GL(T_{y_i}X/T_{y_i}Y_i)\simeq\GL(2,\RR)$
is again injective.} $H_i\hookrightarrow\GL(2,\RR)$, which implies that $m-2\leq 1$.

The other case is that in which each $r_i$ is of the form $p^{e_i}$ for some prime $p$ independent
of $i$. This is dealt with as in the previous case using item (2) of Lemma \ref{lemma:restriction-to-surface},
and again the conclusion is that $m-2\leq 1$.

To conclude the proof in dimension 4, we consider the case in which $X$ is a smooth, closed, connected and non-orientable $4$-manifold.
Let $\pi:Y\to X$ be the orientation covering of $X$. Then $Y$ is a smooth, closed, connected and orientable $4$-manifold,
so the previous arguments, combined with \cite[Theorem 1.12]{Mundet2021} (adapted to smooth actions), imply that
$\discsym_{\smooth}(X)\leq 4$. Furthermore, if $\discsym_{\smooth}(X)=4$, then $\discsym_{\smooth}(Y)=4$, so
$H^*(Y;\ZZ)\simeq H^*(T^4;\ZZ)$ as rings. We are going to see that the assumption $\discsym_{\smooth}(X)=4$ leads
to a contradiction; hence, if $X$ is non-orientable we have $\discsym_{\smooth}(X)\leq 3$.

Assume that $\discsym_{\smooth}(X)=4$, so that there exist integers
$r_i\to\infty$ and a smooth effective action of
$(\ZZ/r_i)^4$ on $X$ for each $i$. As we said, in this case $H^*(Y;\ZZ)\simeq H^*(T^4;\ZZ)$,
so $Y$ is rationally hypertoral and $H^*(Y;\QQ)\simeq\Lambda^*H^1(Y;\QQ)$.

Let $\sigma:Y\to Y$ be the orientation reversing involution satisfying $\pi\circ\sigma=\pi$.
The morphism $\pi^*:H^*(X;\QQ)\to H^*(Y;\QQ)$ is injective and its image can be identified
with the subspace of $\sigma$-invariant classes in $H^*(Y;\QQ)$. Hence $\sigma$ cannot act
trivially on $H^*(Y;\QQ)$, because that would imply that $H^4(X;\QQ)\simeq\QQ$, which is
impossible because $X$ is non-orientable. Since $H^*(Y;\QQ)\simeq\Lambda^*H^1(Y;\QQ)$
it follows that $\sigma$ acts nontrivially on $H^1(Y;\QQ)$, and hence also on $H^1(Y;\ZZ)$.

Arguing as in the proof of \cite[Lemma 7.1]{Mundet2021} and using the assumption that
$(\ZZ/r_i)^4$ acts smoothly and effectively on $X$ for each $i$, we conclude that there are integers
$s_i\to\infty$ and, for each $i$, a smooth and effective action of $(\ZZ/s_i)^4$ on $Y$ that commutes with
the involution $\sigma$. Applying Lemma \ref{lemma:non-trivial-involution} and
Theorems \ref{thm:Minkowski} and \ref{thm:rotation-hypertoral} to the actions
of $(\ZZ/s_i)^4$ on $Y$ we reach a contradiction, so
$\discsym_{\smooth}(X)\leq 3$ if $X$ is non-orientable.
This finishes the proof of the theorem in the 4-dimensional case.

The proof for 3-dimensional manifolds follows the same scheme, but the details are simpler.
The same argument as before reduces the proof to the case of orientation preserving
actions of $(\ZZ/r)^m$ on orientable $3$-manifolds. Orientability implies that the
fixed point set of a finite group action is a disjoint union of a number of copies
of the circle, and the number of copies is bounded above by a constant depending
only on the manifold, by Smith theory. Hence one needs to use the analogue of
Lemma \ref{lemma:restriction-to-surface} where the surface is replaced by the circle.

Alternatively, if $X$ is a smooth, closed and connected $3$-manifold and $(\ZZ/r)^m$
acts smoothly and effectively on $X$ then $(\ZZ/r)^{m+1}$ acts smoothly and effectively
on $Z:=X\times S^1$, where the $(m+1)$-th factor $\ZZ/r$ acts by rotations on the
$S^1$ factor. Applying the $4$-dimensional case of the theorem to $Z$ we deduce the
proof of the theorem for $X$.

\subsection{Proof of Theorem \ref{thm:disc-sym-low-dimension}}
\label{ss:thm:disc-sym-low-dimension}

As in the proof of Theorem \ref{thm:smooth-disc-sym-low-dimension},
the case $\dim X=1$ is elementary, and for the case $\dim X=2$ the comments in item (2) of
Example \ref{ex:computation-discsym} give the result.

Hence we only need to consider the $3$-dimensional case. Assume that
$X$ is a closed topological manifold of dimension $3$. By Moise's theorem \cite{Moise}
(see also \cite{Bing2}), $X$ has a unique smooth structure (see also \cite[Section 3.10]{Thurston}).
By a recent result of Pardon \cite{Pardon2019}
any finite group acting effectively and topologically on $X$ admits effective smooth actions on $X$ (although not
every topological action is conjugate to a smooth action, as illustrated by the famous example due to Bing \cite{Bing}).
Consequently, $\discsym_{\smooth}(X)=\discsym(X)$. So Theorem \ref{thm:smooth-disc-sym-low-dimension} implies
that $\discsym(X)\leq 3$, and that if $\discsym(X)=3$ then $H^*(X;\ZZ)\simeq H^*(T^3;\ZZ)$.

We next prove that if $\discsym(X)=3$ then $X$ is homeomorphic to $T^3$.
By the previous arguments it suffices to prove that if $X$ is a smooth closed $3$-manifold such
that $H^*(X;\ZZ)\simeq H^*(T^3;\ZZ)$ and $\discsym_{\smooth}(X)=3$ then $X$ is diffeomorphic to $T^3$.
By the arguments in \cite[\S 2]{Zimmermann2014}, the fact that $X$ supports smooth effective
actions of arbitrarily large finite groups implies that $X$ supports an effective action of $S^1$.
By Lemma \ref{lemma:free-actions-3-manifold} such action is free, so $X$ is the total space of a
circle bundle on a closed surface $Y=X/S^1$. The surface $Y$ is connected and orientable because $X$ is, as $H^0(X;\ZZ)\simeq H^3(X;\ZZ)\simeq \ZZ$. Consider the following portion of Gysin's exact sequence
of the circle bundle $X\to Y$:
$$0\to H^1(Y;\ZZ)\to H^1(X;\ZZ)\to H^0(Y;\ZZ)\stackrel{\smallsmile e}{\longrightarrow} H^2(Y;\ZZ) $$
where $e\in H^2(Y;\ZZ)$ is the Euler class. Since $H^0(Y;\ZZ)\simeq\ZZ\simeq H^2(Y;\ZZ)$,
if $e\neq 0$ then $H^1(Y;\ZZ)\simeq H^1(X;\ZZ)\simeq\ZZ^3$,
which is impossible for a closed connected and orientable surface $Y$. Hence $e=0$, so $X\cong Y\times S^1$. Again the previous sequence implies that $H^1(Y;\ZZ)\simeq\ZZ^2$, so $Y$ is a $2$-torus. This implies that $X\cong T^3$.

\subsection{Proof of Theorem \ref{thm:nilpotent-generators}}
\label{s:proof-thm:nilpotent-generators}

We will use the following notation. If $G$ is a group, $\Aut(G)$ denotes the group of automorphisms
of $G$. If $G'\leq G$ is an inclusion of groups, $\Aut(G,G')$ denotes
the group of automorphisms $\phi\in\Aut(G)$ such that $\phi(G')=G'$.
If $G,H$ are groups, $\Mor(G,H)$ denotes the set of all group morphisms $G\to H$.
If $H$ is a subgroup of $G$, $N_G(H)$ denotes the normalizer of $H$ in $G$.

\begin{lemma}
\label{lemma:index-automorfismes}
Let $H\leq G$ be an inclusion of finite groups. We have
$$[\Aut(G):\Aut(G,H)]\leq ([G:H]!)^{d(H)+[G:H]},$$
and if $H$ is normal then we have
$[\Aut(G):\Aut(G,H)]\leq [G:H]^{d(H)+[G:H]}$.
\end{lemma}
\begin{proof}
If $H$ is normal, define $K:=H$. Otherwise,
let $K$ be the kernel of the morphism $G\to\Perm(G/H)$
given by left multiplication of $G$ on $G/H$, where $\Perm(S)$ denotes the
group of permutations of the set $S$. In both cases,
$K$ is a normal subgroup of $G$, and if $H$ is not normal
then $[G:K]\leq [G:H]!$.

Let $\sigma:\Aut(G)\to\Mor(G,G/K)$ be the map sending $\phi\in\Aut(G)$
to $\pi\circ\phi$, where $\pi:G\to G/K$ is the projection.
We have $d(G)\leq d(H)+[G:H]$. A morphism $G\to G/K$ is uniquely
determined by the images of the elements in a generating set of $G$,
so $|\Mor(G,G/K)|\leq L:=|G/K|^{d(H)+[G:H]}$.
Hence, there is some $\phi\in\Aut(G)$ such that $|\sigma^{-1}(\sigma(\phi))|\geq |\Aut(G)|/L$.
Let $\psi\in \sigma^{-1}(\sigma(\phi))$, and write
$\psi=\phi\circ\xi$ for some $\xi\in\Aut(G)$. Let $h:G\to G$ be the map defined
by $h(g)=\xi(g)g^{-1}$. We have $\sigma(\psi)=\sigma(\phi)$, so
$\phi(\xi(g)g^{-1})=\phi(\xi(g))\phi(g^{-1})=\psi(g)\phi(g)^{-1}\in K$ for every $g\in G$,
or equivalently $\xi(g)g^{-1}\in\phi^{-1}(K)$ for every $g$. In particular, if
$g\in\phi^{-1}(H)$ then $\xi(g)\in g\phi^{-1}(K)\subseteq \phi^{-1}(H)$,
so $\xi(\phi^{-1}(H))=\phi^{-1}(H)$.
Hence the image of the injective map
$$\sigma^{-1}(\sigma(\phi))\ni\psi\mapsto \phi^{-1}\circ\psi\in\Aut(G)$$
is contained in $\Aut(G,\phi^{-1}(H))$, so $|\Aut(G,\phi^{-1}(H))|\geq |\sigma^{-1}(\sigma(\phi))|
\geq |\Aut(G)|/L$.
But $\Aut(G,H)\ni\theta\mapsto \phi^{-1}\circ\theta\circ\phi\in\Aut(G,\phi^{-1}(H))$ is a bijection,
so the lemma follows.
\end{proof}

Fix natural numbers $k,C$.
We claim that there is a constant $\Lambda$ such that any finite $p$-group $P\in\nN_{k,C}$
has a subgroup $P'\leq P$ satisfying $[P:P']\leq \Lambda$ and $d(P')\leq k(5k+1)/2$.

Suppose that $P\in\nN_{k,C}$ is a finite $p$-group. Let $A$ be a maximal
abelian normal subgroup of $P$. There is a subgroup $B\leq A$ satisfying $d(B)\leq k$
and $[A:B]\leq C$. By Lemma \ref{lemma:index-automorfismes},
$[\Aut(A):\Aut(A,B)]\leq C^{k+C}$.
Let $\rho:\Aut(A,B)\to\Aut(B)$ be the restriction map.
The kernel of the natural morphism
$\eta:\Ker\rho\to\Aut(A/B)$ is equal to
$\{\Id_A+\psi\circ\pi\mid \psi\in\Mor(A/B,B)\}$, where
$\pi:A\to A/B$ is the projection (recall that we use additive notation on abelian groups).
The map $\Id_A+\psi\circ\pi\mapsto \psi$
gives an isomorphism of groups $\Ker\eta\simeq\Mor(A/B,B)$, where
the group structure on $\Mor(A/B,B)$
is inherited by the group structure on $B$.
Let $e\in\ZZ$ satisfy $p^e\leq C<p^{e+1}$. Since
$A$ is a $p$-group, $|A/B|\leq p^e$. The $p^e$-torsion $B[p^e]\leq B$
satisfies $|B[p^e]|\leq p^{ed(B)}\leq p^{ek}\leq C^k$.
Since $\Mor(A/B,B)=\Mor(A/B,B[p^e])$, we have
$|\Mor(A/B,B)|\leq (C^k)^C=C^{kC}.$
Hence:
$$|\Ker\rho|\leq |\Aut(A/B)|\cdot |\Mor(A/B,B)|\leq C!C^{kC}.$$
The action of $P$ by conjugation on itself induces a morphism $\zeta:P\to \Aut(A)$
whose kernel is equal to $A\leq P$ (see e.g. \cite[\S 5.2.3]{Robinson}).
Let $P_0=\zeta^{-1}(\Aut(A,B))$. Then
$$[P:P_0]\leq [\Aut A:\Aut(A,B)]\leq C^{k+C}.$$
Since $d(B)\leq k$, the Gorchakov--Hall--Merzlyakov--Roseblade lemma (see e.g. \cite[Lemma 5]{Roseblade}) implies that the subgroup $\rho(\zeta(P_0))\leq\Aut(B)$ satisfies
$d(\rho(\zeta(P_0)))\leq k(5k-1)/2$. Hence we may pick up elements $g_1,\dots,g_r\in P_0$, with $r\leq k(5k-1)/2$, such that $\rho(\zeta(g_1)),\dots,\rho(\zeta(g_r))$ generate $\rho(\zeta(P_0))$.
Let $P'\leq P_0$ be the subgroup generated by the elements $g_1,\dots,g_r\in P_0$ and by $B$. Clearly $d(P')\leq k+k(5k-1)/2=k(5k+1)/2$.

We now bound $[P:P']$. From the exact sequence
$$1\to \zeta(P_0)\cap\Ker\rho\to \zeta(P_0)\to \rho(\zeta(P_0))\to 1$$
we conclude that
$$|\zeta(P_0)|\leq |\zeta(P_0)\cap\Ker\rho|\cdot |\rho(\zeta(P_0))|
\leq |\Ker\rho|\cdot |\rho(\zeta(P_0))|\leq C!C^{kC}|\rho(\zeta(P_0))|.$$
Since $\rho:\zeta(P')\to \rho(\zeta(P_0))$ is surjective, we have $|\zeta(P')|\geq |\rho(\zeta(P_0))|$.
The two estimates imply $[\zeta(P_0):\zeta(P')]\leq C!C^{kC}$.
We have $\Ker \zeta\cap P_0=A$ and $\Ker \zeta\cap P'=B$, so $[P_0:P']=[A:B][\zeta(P_0):\zeta(P')]\leq C\cdot C!C^{kC}$.
Combining this with our estimate on $[P:P_0]$ we obtain
$$[P:P']=[P:P_0]\cdot [P_0:P']\leq \Lambda:=C^{k+C}\cdot C\cdot C!C^{kC}=C^{(k+1)(C+1)}C!.$$
This finishes the proof of the claim.

Now let $N\in\nN_{k,C}$ an arbitrary group.
Let $p_1<\dots<p_s$ be the primes dividing $|N|$, and for each $i$
let $P_i\leq N$ be a Sylow $p_i$-subgroup of $N$. By \cite[Theorem 1.26]{Isaacs}
the multiplication map $\mu:P_1\times\dots\times P_s\to N$,
$\mu(p_1,\dots,p_s)=p_1\dots p_s$, is an isomorphism of groups.

By the claim, for each $i$ there is a subgroup $P_i'\leq P_i$ satisfying
$[P_i:P_i']\leq \Lambda$ and $d(P_i')\leq \delta:=k(5k+1)/2$. Let
$N'=\mu(P_1'\times\dots\times P_s')$.
We claim that $d(N)\leq \delta$. To prove the claim choose, for each $1\leq i\leq s$,
elements $e_{i1},\dots,e_{i\delta}\in P_i'$ generating $P_i'$ (some repetitions might
be necessary).
Let $e_j:=e_{1j}\dots e_{sj}$. Since the elements $e_{1j},\dots,e_{sj}$ pairwise commute
and have pairwise coprime order, the Chinese remainder theorem implies that for each $i,j$
there is some power of $e_j$ which is equal to $e_{ij}$. This proves that
$e_1,\dots,e_{\delta}$ generates $N'$, so the claim is proved.

If $p_i>\Lambda$ then the condition $[P_i:P_i']\leq\Lambda$ implies that $P_i'=P_i$. Since
the number of primes in $\{1,\dots,\Lambda\}$ is at most $\Lambda$, it follows that
$[N:N']=\prod_i [P_i:P_i']\leq \Lambda^{\Lambda}$.

\section*{Acknowledgments}
It is for me a pleasure to dedicate this paper to Oscar Garc\'{\i}a--Prada
on the occasion of this 60th birthday. I was very privileged to be Oscar's
PhD student twenty five years ago, and I can't overstate how many things I have
learned from him since that time, both in mathematics and beyond.

Many thanks to Constantin Shramov, Endre Szab\'o and Alexandre Turull for
helpful comments and corrections, and also to the referee for a very detailed
report with many suggestions that helped a lot to improve the text.

This research was partially supported by the grant
PID2019-104047GB-I00 from the Spanish Ministeri de Ci\`encia i Innovaci\'o.


\begin{thebibliography}{0}

\bibitem{AbouBlum}
Mohammed Abouzaid, Andrew J. Blumberg, Arnold Conjecture and
Morava K-theory, arXiv:2103.01507

\bibitem{AP} C. Allday, V. Puppe,
{\it Cohomological methods in transformation groups},
Cambridge Studies in Advanced Mathematics, 32. Cambridge University Press, Cambridge, 1993.

\bibitem{AssadiBurghelea}
A. Assadi, D. Burghelea,
Examples of asymmetric differentiable manifolds.
{\em Math. Ann.} {\bf 255} (1981), no. 3, 423--430.

\bibitem{BaiXu}
Shaoyun Bai, Guangbo Xu, Arnold conjecture over integers, arXiv:2209.08599

\bibitem{Baumgartner}
J.C. Baumgartner, Kohomologie freier $(\ZZ_p)^r$ - R\"aume, Dissertation, Kostanz, 1990.

\bibitem{BCG}
G. Besson, G. Courtois, S. Gallot,
Minimal entropy and Mostow's rigidity theorems,
{\em Ergodic Theory Dynam. Systems} {\bf 16} (1996), no. 4, 623--649.

\bibitem{Bing}
R.H. Bing,
A homeomorphism between the 3-sphere and the sum of two solid horned spheres.
{\em Ann. of Math.} (2) {\bf 56} (1952), 354--362.

\bibitem{Bing2}
R.H. Bing,
An alternative proof that 3-manifolds can be triangulated,
{\em Ann. of Math.} (2) {\bf 69} (1959), 37--65.

\bibitem{Bloomberg}
E.M. Bloomberg,
Manifolds with no periodic homeomorphisms,
{\em Trans. Amer. Math. Soc.} {\bf 202} (1975), 67--78.

\bibitem{BKKY22}
F. Bogomolov, N. Kurnosov, A. Kuznetsova, E. Yasinsky,
Geometry and automorphisms of non-Kähler holomorphic symplectic manifolds,
{\em Int. Math. Res. Not. IMRN} (2022), no.16, 12302--12341.



\bibitem{Bo} A. Borel,
{\em Seminar on transformation groups}, Ann. of Math. Studies {\bf 46},
Princeton University Press, N.J., 1960.


\bibitem{BGPG}
S.B. Bradlow, O. Garc\'{\i}a-Prada, P.B. Gothen,
Surface group representations and U(p,q)-Higgs bundles,
{\em J. Differential Geom.} {\bf 64} (2003), no. 1, 111--170.

\bibitem{Bredon}
G.E. Bredon,
{\em Introduction to compact transformation groups},
Pure and Applied Mathematics, Vol. {\bf 46}, Academic Press, New York-London (1972).

\bibitem{Breu}
E. Breuillard, An exposition of Jordan's original proof of his theorem on finite
subgroups of $\GL_n(\CC)$, see \\ {\tt https://www.imo.universite-paris-saclay.fr/\~{ }breuilla/Jordan.pdf}

\bibitem{CWY}
S. Cappell, S. Weinberger, M. Yan,
Closed aspherical manifolds with center,
{\em J. Topol.} {\bf 6} (2013), no. 4, 1009--1018.

\bibitem{Carlsson}
G. Carlsson, On the homology of finite free $(\ZZ/2)^n$-complexes,
{\em Invent. Math.} {\bf 74} (1983), no. 1, 139--147.

\bibitem{CRW}
P.E. Conner, F. Raymond, P.J. Weinberger,
Manifolds with no periodic maps. Proceedings of the Second Conference on Compact Transformation Groups (Univ. Massachusetts, Amherst, Mass., 1971), Part II, pp. 81–108. Lecture Notes in Math., Vol. 299, Springer, Berlin, 1972.

\bibitem{ConstantinKolev}
A. Constantin, B. Kolev,
The theorem of Ker\'ekj\'art\'o on periodic homeomorphisms of the disc and the sphere,
{\em Enseign. Math.} (2) {\bf 40} (1994), no. 3-4, 193--204.

\bibitem{CPS} B. Csik\'os, L. Pyber, E. Szab\'o, Diffeomorphism
    groups of compact $4$-manifolds are not always Jordan,
    preprint {\tt arXiv:1411.7524}.

\bibitem{CPS2} B. Csik\'os, L. Pyber, E. Szab\'o, Finite subgroups of the homeomorphism group
of a compact manifold are almost nilpotent, {\em preprint} {\tt arXiv:2204.13375}.

\bibitem{CMPS}
B. Csik\'os, I. Mundet i Riera, L. Pyber, E. Szab\'o, Number of stabilizer subgroups in a finite
group acting on a manifold, {\it preprint} {\tt arXiv:2111.14450}.

\bibitem{CR} C.W. Curtis, I. Reiner, {\em Representation Theory
    of Finite Groups and Associative Algebras}, reprint of the
    1962 original, AMS Chelsea Publishing, Providence, RI,
    2006.


\bibitem{MWD} M.W. Davis, A survey of results in higher dimensions, Chapter XI in
{\em The Smith conjecture}. Papers presented at the symposium held at Columbia University, New York, 1979. Edited by John W. Morgan and Hyman Bass. Pure and Applied Mathematics, 112. Academic Press, Inc., Orlando, FL, 1984.

\bibitem{DH}
R.M. Dotzel, G.C. Hamrick,
$p$-group actions on homology spheres,
{\em Invent. Math.} {\bf 62} (1981) 437--442.


\bibitem{Edmonds1985}
A.L. Edmonds,
Transformation groups and low-dimensional manifolds. Group actions on manifolds (Boulder, Colo., 1983), 339--366,
Contemp. Math., 36, Amer. Math. Soc., Providence, RI, 1985.

\bibitem{Edmonds2018}
A.L. Edmonds,
A survey of group actions on 4-manifolds,
{\em Handbook of group actions.} Vol. III, 421--460,
Adv. Lect. Math. (ALM), 40, Int. Press, Somerville, MA, 2018.


\bibitem{Fisher}
D. Fisher,
Groups acting on manifolds: around the Zimmer program,
{\em Geometry, rigidity, and group actions}, 72–157,
Chicago Lectures in Math., Univ. Chicago Press, Chicago, IL (2011).

\bibitem{FrTa}
A. Fr\"ohlich, M.J. Taylor,
Algebraic number theory. Cambridge Studies in Advanced Mathematics, 27.
Cambridge University Press, Cambridge, 1993.

\bibitem{Ghys1993}
\'E. Ghys,
Sur les groupes engendr\'es par des diff\'eomorphismes proches de l'identit\'e,
{\em Bol. Soc. Brasil. Mat.} (N.S.) {\bf 24} (1993), no. 2, 137--178.

\bibitem{Ghys}
\'E. Ghys, The following talks:
{\em Groups of diffeomorphisms},
Col\'oquio brasileiro de matem\'aticas, Rio de Janeiro (Brasil), July 1997;
{\em The structure of groups acting on manifolds},
Annual meeting of the Royal Mathematical Society, Southampton (UK), March 1999;
{\em Some open problems concerning group actions},
Groups acting on low dimensional manifolds, Les Diablerets (Switzerland), March 2002;
{\em Some Open problems in foliation theory},
Foliations 2006, Tokyo (Japan), September 2006.

\bibitem{Gh2}
\'E. Ghys, Talk at IMPA: {\it My favourite groups}, April 2015.

\bibitem{Golota}
A. Golota, Finite abelian subgroups in the groups of birational and bimeromorphic selfmaps,
{\it preprint} {\tt arXiv:2205.00607}.

\bibitem{Golota2}
A. Golota, Finite groups acting on compact complex parallelizable manifolds,
{\it preprint} {\tt arXiv:2302.13513}.

\bibitem{Guld}
A. Guld, Finite subgroups of the birational automorphism group are "almost" nilpotent of class at most two, {\em preprint} {\tt arXiv:2004.11715}.


\bibitem{Hambleton}
I. Hambleton,
Topological spherical space forms. {\em Handbook of group actions.} Vol. II, 151--172,
Adv. Lect. Math. (ALM), 32, Int. Press, Somerville, MA, 2015.

\bibitem{Hanke}
B. Hanke,
The stable free rank of symmetry of products of spheres,
{\em Invent. Math.} {\bf 178} (2009), no. 2, 265--298.
Erratum to: The stable free rank of symmetry of products of spheres,
{\em Invent. Math.} {\bf 182} (2010), no. 1, 229.

\bibitem{HKMS}
R. Haynes, S. Kwasik, J. Mast, R. Schultz,
{\em Periodic maps on $\RR^7$ without fixed points},
Math. Proc. Cambridge Philos. Soc. {\bf 132} (2002), no. 1, 131--136.

\bibitem{Hernandez}
L. Hernández,
Maximal representations of surface groups in bounded symmetric domains,
{\em Transactions Amer. Math. Soc.} {\bf 324} (1991) 405--420.


\bibitem{Hsiang}
W.-y. Hsiang,
{\em Cohomology theory of topological transformation groups},
Ergebnisse der Mathematik und ihrer Grenzgebiete, Band 85. Springer-Verlag, New York-Heidelberg, 1975.

\bibitem{Isaacs}
I.M Isaacs, {\em Finite group theory},
Graduate Studies in Mathematics, 92. American Mathematical Society, Providence, RI, 2008.

\bibitem{Jones}
L. Jones
The converse to fixed point theorem of P.A. Smith I
{\em Ann. Math.}  {\bf 94} (1971), pp. 52--68.

\bibitem{Jordan}
C. Jordan, M\'emoire sur les \'equations diff\'erentielles lin\'eaires \`a int\'egrale alg\'ebrique,
{\em J. Reine Angew. Math.} {\bf 84} (1878) 89--215.

\bibitem{Kim}
J.H. Kim,
Jordan property and automorphism groups of normal compact K\"ahler varieties,
Commun. Contemp. Math. {\bf 20} (2018), no. 3, 1750024, 9 pp.

\bibitem{Kreck}
M. Kreck, Simply connected asymmetric manifolds,
{\em J. Topol.} {\bf 2} (2009), no. 2, 249--261.
Corrigendum: Simply connected asymmetric manifolds,
{\em J. Topol.} {\bf 4} (2011), no. 1, 254--255.


\bibitem{Log22}
K. Loginov,
Jordan property for groups of bimeromorphic self-maps of complex manifolds with large Kodaira dimension,
{\it preprint}
{\tt arXiv:2209.12032}.

\bibitem{MannSu} L. N. Mann, J. C. Su, Actions of elementary
    p-groups on manifolds, {\em Trans. Amer. Math. Soc.} {\bf
    106} (1963), 115--126.

\bibitem{MPZ}
S. Meng, F. Perroni, D.-Q. Zhang,
Jordan property for automorphism groups of compact spaces in Fujiki's class ${\mathcal C}$,
{\em J. Topol.} {\bf 15} (2022), no. 2, 806--814.

\bibitem{MZ}
S. Meng, D.-Q. Zhang, Jordan property for non-linear algebraic groups and
projective varieties, {\em Amer. J. Math.} {\bf 140} (2018), no. 4, 1133--1145.

\bibitem{Minkowski} H. Minkowski, Zur Theorie der positiven quadratischen Formen, {\em Journal fur die reine und angewandte Mathematik} {\bf 101} (1887), 196-202.

\bibitem{Moise}
E.E. Moise,
Affine structures in 3-manifolds. V. The triangulation theorem and Hauptvermutung,
{\em Ann. of Math.} (2) {\bf 56} (1952), 96--114.

\bibitem{Mundet2010} I. Mundet i Riera, Jordan's theorem for the
    diffeomorphism group of some manifolds {\em Proc. Amer.
    Math. Soc.} {\bf 138} (2010), 2253-2262.

\bibitem{Mundet2016} I. Mundet i Riera, Finite group actions on 4-manifolds with nonzero Euler
characteristic, {\em Math. Z.} {\bf 282} (2016), 25--42.

\bibitem{Mundet2017-0}
I. Mundet i Riera,
Finite groups acting symplectically on $T^2\times S^2$,
{\em Trans. Amer. Math. Soc.} {\bf 369} (2017), no. 6, 4457--4483.

\bibitem{Mundet2017} I. Mundet i Riera,
Non Jordan groups of diffeomorphisms and actions of compact Lie groups on manifolds,
{\it Transformation Groups} {\bf 22} (2017), no. 2, 487--501.

\bibitem{Mundet2018} I. Mundet i Riera,
Finite subgroups of Ham and Symp.,
{\em Math. Ann.} {\bf 370} (2018), no. 1--2, 331--380.


\bibitem{Mundet2019}
I. Mundet i Riera, Finite group actions on homology spheres and manifolds with nonzero Euler characteristic,
{\em J. Topol.} {\bf 12} (2019), no. 3, 744--758.

\bibitem{Mundet2020}
I. Mundet i Riera,
Isometry groups of closed Lorentz 4-manifolds are Jordan,
{\em Geom. Dedicata} {\bf 207} (2020), 201--207.


\bibitem{Mundet2021}
I. Mundet i Riera,
Discrete degree of symmetry of manifolds,
{\em preprint} {\tt arXiv:2112.05599v2}.

\bibitem{Mundet2022}
I. Mundet i Riera,
Jordan property for homeomorphism groups and almost fixed point property,
{\it preprint} {\tt arXiv:2210.07081}.

\bibitem{Mundet2023-0}
I. Mundet i Riera,
Finite abelian group actions on some topological manifolds, in preparation.

\bibitem{Mundet2023-1}
I. Mundet i Riera,
Free actions of finite abelian groups on topological manifolds, in preparation.

\bibitem{MundetSaez}
I. Mundet i Riera, C. S\'aez-Calvo,
Which finite groups act smoothly on a given 4-manifold?
{\em Trans. Amer. Math. Soc.} {\bf 375} (2022), no. 2, 1207--1260.

\bibitem{MundetTurull}
I. Mundet i Riera, A. Turull,
Boosting an analogue of Jordan's theorem for finite groups,
{\em Adv. Math.} {\bf 272} (2015), 820--836.

\bibitem{Pardon2019}
J. Pardon, Smoothing finite group actions on three-manifolds,
{\em Duke Math. J.} {\bf 170} (2021), no. 6, 1043--1084.


\bibitem{Po0} V.L. Popov, On the Makar-Limanov, Derksen
    invariants, and finite automorphism groups of algebraic
    varieties. In {\em Peter Russell’s Festschrift, Proceedings of
    the conference on Affine Algebraic Geometry held in
    Professor Russell’s honour}, 1–5 June 2009, McGill Univ.,
    Montreal., volume 54 of Centre de Recherches Math\'ematiques
    CRM Proc. and Lect. Notes, pages 289–311, 2011.

\bibitem{PS14}
Yu. Prokhorov, C. Shramov,
Jordan property for groups of birational selfmaps,
{\em Compos. Math.} {\bf 150} (2014), no. 12, 2054--2072.

\bibitem{PS16}
Yu. Prokhorov, C. Shramov,
Jordan property for Cremona groups,
{\em Amer. J. Math.} {\bf 138} (2016), no. 2, 403--418.

\bibitem{PS21}
Yu. Prokhorov, C. Shramov,
Automorphism groups of compact complex surfaces,
{\em  Int. Math. Res. Not. IMRN} (2021), no. 14, 10490–10520.

\bibitem{Robinson}
D.J.S. Robinson,
{\em A course in the theory of groups},
Second edition, Graduate Texts in Mathematics {\bf 80},
Springer-Verlag, New York, 1996.

\bibitem{Puppe}
V. Puppe,
Do manifolds have little symmetry?
{\em J. Fixed Point Theory Appl.} {\bf 2} (2007), no. 1, 85--96.

\bibitem{Rezchikov}
Semon Rezchikov, Integral Arnol'd Conjecture, arXiv:2209.11165

\bibitem{Roseblade}
J.E. Roseblade,
On groups in which every subgroup is subnormal.
{\em J. Algebra} {\bf 2} (1965) 402--412.

\bibitem{Sav24}
A. Savelyeva,
Automorphisms of Hopf manifolds,
J. Algebra {\bf 638} (2024), 670--681.

\bibitem{Sch} R. Schultz,
Nonlinear analogs of linear group actions on spheres,
{\em Bull. Amer. Math. Soc.} {\bf 11} (1984), 263--285.

\bibitem{Serre}
J.-P. Serre,
A Minkowski-style bound for the orders of the finite subgroups of the Cremona group of
rank 2 over an arbitrary field,
{\em Mosc. Math. J.} {\bf 9} (2009), no. 1, 193--208.

\bibitem{Smith}
P.A. Smith,
Permutable periodic transformations,
{\em Proc. Nat. Acad. Sci. U.S.A.} {\bf 30} (1944), 105--108.

\bibitem{Szabo2019}
D.R. Szab\'o, Special $p$-groups acting on compact manifolds, {\it preprint} {\tt arXiv:1901.07319}.

\bibitem{tD}
T. tom Dieck, {\em Transformation Groups}, De Gruyter Studies in Mathematics, vol. {\bf 8},
Walter de Gruyter \& Co., Berlin, 1987.

\bibitem{Thurston}
W.P. Thurston,
Three-dimensional geometry and topology. Vol. 1.
Edited by Silvio Levy. Princeton Mathematical Series, {\bf 35}. Princeton University Press, Princeton, NJ, 1997.

\bibitem{Toledo}
D. Toledo, Representations of surface groups in complex hyperbolic space,
{\em J. Differential Geometry} {\bf 29} (1989) 125--133.

\bibitem{vL}
W. van Limbeek,
Symmetry gaps in Riemannian geometry and minimal orbifolds,
{\em J. Differential Geom.} {\bf 105} (2017), no. 3, 487--517.

\bibitem{Ye2}
S. Ye, Euler characteristics and actions of automorphism groups of free groups,
{\em Algebr. Geom. Topol.} {\bf 18} (2018), 1195--1204.

\bibitem{Ye}
S. Ye, Symmetries of flat manifolds, Jordan property and the general Zimmer program,
J. Lond. Math. Soc. (2) {\bf 100} (2019), no. 3, 1065--1080.

\bibitem{Zarhin}
Y.G. Zarhin, Jordan groups and elliptic ruled surfaces,
{\em Transform. Groups} {\bf 20} (2015), no. 2, 557--572.



\bibitem{Zimmermann2014}
B.P. Zimmermann, On Jordan
    type bounds for finite
    groups acting on compact $3$-manifolds,
    {\em Arch. Math.} {\bf 103} (2014), 195--200.

\bibitem{Zimmermann2017}
B. Zimmermann,
On topological actions of finite, non-standard groups on spheres,
{\em Monatsh. Math.} {\bf 183} (2017), no. 1, 219--223.

\end{thebibliography}
\end{document}